\documentclass[12pt]{amsart}

\usepackage{
amsfonts,
latexsym,
amssymb,
mathabx,
enumerate,
verbatim,
mathrsfs,
}

\newcommand{\labbel}{\label}

\newtheorem{theorem}{Theorem}[section]
\newtheorem{lemma}[theorem]{Lemma}

\newtheorem{proposition}[theorem]{Proposition} 
 
\newtheorem{corollary}[theorem]{Corollary}

\newtheorem*{theorem*}{Theorem}
\newtheorem*{corollary*}{Corollary}

\theoremstyle{definition}
\newtheorem{definition}[theorem]{Definition}
\newtheorem{definitions}[theorem]{Definitions}

\newtheorem{problems}[theorem]{Problems} 

\theoremstyle{remark}
\newtheorem{remark}[theorem]{Remark}
 
\newtheorem{example}[theorem]{Example}

\newtheorem{notation}[theorem]{Notation} 
 
\newtheorem*{acknowledgement}{Acknowledgement}
\newtheorem*{disclaimer}{Disclaimer}

\newcommand{\+}{\mathbin{\hash}}
\DeclareMathOperator*{\pl}{\mathbin{\hash}}
\DeclareMathOperator*{\nsum}
    {\ifdim\displaywidth>0pt {{\mbox{\Large \#}}}
     \else{{\mbox{\large \#}}}\fi} 

\newcommand{\sumh}{\sum^{\#}}
\newcommand{\sumf}{\sum^{F}}
\newcommand{\sumg}{\sum^{G}}
\newcommand{\sumhh}{\sum^{H}}

\newcommand{\sh}{S}
\newcommand{\snorm}{S^{\Sigma}}

 \changenotsign  

 \allowdisplaybreaks[1]

\begin{document}
 
\title
{Some transfinite natural sums}

\author{Paolo Lipparini} 
\address{Doppiomento di Matematica\\Viale della  Ricerca Scientifica\\II Universit\`a di Roma (Tor Vergata)\\I-00133 ROME ITALY}
\urladdr{http://www.mat.uniroma2.it/\textasciitilde lipparin}


\thanks{Work performed under the auspices of G.N.S.A.G.A. Work 
partially supported by PRIN 2012 ``Logica, Modelli e Insiemi''}

\begin{abstract}
We study a transfinite iteration 
of the ordinal Hessenberg natural sum
obtained by taking suprema at limit stages and
 show that such an iterated natural sum
differs from 
 the more usual transfinite ordinal sum
only for  a finite number of iteration steps.
The iterated natural sum of a sequence of ordinals
can be obtained as a
``mixed sum''
(in an order-theoretical sense) 
of the ordinals in the sequence; in fact, it is the largest mixed sum
which satisfies a finiteness condition, relative to the ordering
of the sequence.
We introduce other infinite natural sums which are  invariant under
permutations and show that they all coincide in the countable case.
 Finally, in the last section we use the above infinitary natural sums
in order to provide a definition of size for a well-founded tree,
together with an order-theoretical characterization in the countable case.
The proof of this order-theoretical characterization is mostly 
independent from
the rest of this paper. 
\end{abstract} 

\maketitle

\section{Introduction} \labbel{intro} 

The \emph{(Hessenberg) natural sum} $\alpha\+ \beta $ 
of two ordinals $\alpha$ and $\beta$ 
can be defined 
by expressing $\alpha$ and $\beta$
in Cantor normal form and summing linearly. 
Further details shall be given below.
In contrast with the more usual
ordinal sum, 
the resulting natural operation $\+$ is commutative, associative and cancellative.
The definition can be obviously extended
to deal with a finite sum of ordinals; otherwise, using the fact 
that the binary operation $\+$
is commutative and associative, we get
 no ambiguity in defining the natural sum of a  finite 
sequence of ordinals.

An infinitary version of the natural sum 
has appeared in
Wang   \cite{W} and
 V\"a\"an\"anen and Wang   
\cite{VW}
in the countable case
and we have  extensively studied it in \cite{w},
in particular, providing some order-theoretical characterization. 

Here we extend the above countable natural sum
to the transfinite. Expanding on  the countable case, we take suprema
at limit stages and the natural sum at successor steps.
There are similarities with the 
countable case: the transfinitely iterated natural sum 
can be computed in a way similar to the more usual transfinite ordinal sum,
except just for a finite number of  steps, in which we
should take  the finite natural sum in place of the ordinal sum.
Moreover, in the same spirit of 
\cite{w}, the iterated natural sum  has an order-theoretical characterization:
it is the largest mixed sum
 which satisfies a finiteness condition.

On the other hand, in the uncountable case, 
the iterated natural sum 
turns out to be not invariant under permutations.
In particular, the above mentioned order theoretical characterization
depends on the ordering of the ordinals in the sequence.
Some possible definitions of invariant infinitary operations
will be  given in Section \ref{furth}.
 Significantly, 
all these operations coincide in the case of $ \omega$-indexed sequences
of ordinals; this fact is  not completely trivial, 
indeed the corresponding result would be false in the case
of infinitary operations associated with 
the more usual ordinal sum. 

Finally, we use the above infinitary natural sums
in order to provide some definitions of size for well-founded trees,
together with an order-theoretical characterization in the countable case.
This is presented in Section \ref{not}.
For the most part, Section \ref{not} can be read independently from
the rest of this paper; some familiarity
  with \cite{w} is sufficient.

As a final comment, it should be mentioned that, quite remarkably, 
the special case of the iterated natural sum in which all 
summands are equal has been introduced more than a century ago 
by  Jacobsthal \cite{J} in 1909.
 See Altman \cite{A} for further details, references and
generalizations. To the best of our knowledge,
the general case of arbitrary summands has never
been considered before \cite{VW, W}, but it should be mentioned
that the literature on the subject is so vast that
a thorough check is virtually impossible.

\section{Preliminaries} \labbel{prel} 

Recall that every nonzero ordinal $ \alpha $ can be 
expressed in \emph{Cantor normal form} in
a unique way  as follows. 
\begin{equation*} \labbel{cantor}   
 \alpha  =
 \omega ^ {\xi_k} r_k + \omega ^ {\xi _{k-1}}r _{k-1}    +
\dots
+ \omega ^ {\xi_1} r_1 + \omega ^ {\xi_0}r_0  
 \end{equation*}
  for some  integers 
$k \geq 0$, $r_k, \dots, r_0 >0$ and ordinals
$ \xi _k > \xi _{k-1} > \dots > \xi_1 > \xi_0  $.
The ordinal $0$ can be considered as an ``empty'' sum.  
Here sums, products and exponentiations are always considered in
 the ordinal sense.

\begin{definition} \labbel{natsum}
The natural sum $\alpha \+ \beta $ of two ordinals 
$\alpha$ and $\beta$ 
is the only operation
satisfying
\begin{equation*}    
 \alpha \+ \beta  =
 \omega ^ {\xi_k} (r_k + s_k) +
\dots
+ \omega ^ {\xi_1} (r_1+s_1) + \omega ^ {\xi_0}(r_0 + s_0) 
 \end{equation*}
whenever
\begin{equation*}
\begin{split}   
 \alpha   =
 \omega ^ {\xi_k} r_k  +
\dots
+ \omega ^ {\xi_1} r_1 + \omega ^ {\xi_0}r_0\\  
 \beta  =
 \omega ^ {\xi_k}  s_k +
\dots
+ \omega ^ {\xi_1} s_1 + \omega ^ {\xi_0} s_0 
 \end{split} 
\end{equation*}
and $k , r_k, \dots, r_0, s_k, \dots, s_0  < \omega$, 
$ \xi _k > \dots > \xi_1 > \xi_0  $.
 \end{definition}   

The definition is justified by the fact that we can represent every
nonzero $\alpha$ and $\beta$ in Cantor normal form and then insert
some more null coefficients for convenience
in order to make the indices match.
The null coefficients do not affect $\alpha$ and $\beta$,
hence the definition is well-posed.
See, e.~g., Bachmann \cite{Bac} and Sierpi\'nski \cite{Sier} for further details.

Let $\alpha$ and 
$\eta$ be ordinals,
and express $\alpha$ in Cantor normal 
form as
$
 \omega ^ {\xi_k} r_k + 
\dots
+ \omega ^ {\xi_0}r_0$.
The ordinal $\alpha^{\restriction  \eta}$, in words,  
\emph{$\alpha$ truncated at the $\eta^{th}$ exponent of $ \omega$}, 
is
$
 \omega ^ {\xi_k} r_k + 
\dots
+ \omega ^ {\xi_ \ell}r_\ell$,
where $\ell$ is the smallest index
such that $\xi_\ell \geq \eta$.   
We set
$\alpha^{\restriction  \eta}=0$
in case that $\alpha < \omega ^ \eta$.  

\begin{proposition} \labbel{facts}
Let $\alpha$, $\beta$, $\gamma$,
$\eta$ be ordinals and $ r< \omega$. 
  \begin{enumerate}[(1)]    
\item
The operation $\+ $ is commutative, associative, both left and right cancellative
and strictly monotone in both arguments.
\item
$ \sup \{ \alpha, \beta \} \leq  \alpha + \beta \leq \alpha \+ \beta $.
\item \labbel{factsex4}
If $\beta < \omega ^ \eta $, then  $\alpha \+ \beta < \alpha  + \omega ^ \eta $.
\item \labbel{factsn1n2}
If  $\alpha \+ \beta \geq \omega ^ \eta r $,
then there are $r_1, r_2 < \omega $
such that  $r_1 + r_2= r$,
$\alpha  \geq \omega ^ \eta r_1 $
and $\beta \geq \omega ^ \eta r_2 $.
\item\labbel{facts5}
$\alpha+ \beta = \alpha^{\restriction  \eta}+ \beta =
\alpha^{\restriction  \eta}\+ \beta $,
where $\eta$ is the leading exponent in the 
Cantor normal expression of $\beta$. 
\item \labbel{facts3op}
$( \alpha + \beta ) \+ \gamma \geq 
\alpha + (\beta  \+ \gamma).$
\item \labbel{facts7u}
$ \alpha \+ (\beta  + \gamma) \geq 
(\alpha \+ \beta) + \gamma.$
  \end{enumerate}
 \end{proposition} 

\begin{proof}
Everything follows  easily by expressing
the relevant ordinals in Cantor normal form and applying the
definitions. See \cite{Bac,Sier} for details.
A proof of  (\ref{factsex4}) can be found in \cite[Proposition 2.2(4)]{w}. 

In order to prove 
(\ref{factsn1n2}), 
express both $\alpha$ and $\beta$ in Cantor normal form.
If the largest power of $ \omega$ in the expression of, say, $\alpha$ is
$> \eta$, then  $\alpha \geq \omega ^ \eta r $ and we are done,
taking $r_2=0$.
The same holds for $\beta$.
Otherwise, 
$\alpha  = \omega ^ \eta s + \dots $
and $\beta = \omega ^ \eta t + \dots$
(possibly, $s=0$ or $t=0$).
By the definition of $\alpha \+ \beta$, and  
since 
$\alpha \+ \beta \geq \omega ^ \eta r$,
we get
$s+t \geq r$,
hence we can find  
$r_1 $
and
$r_2 $
as desired.

(\ref{facts5}) When computing
$\alpha+ \beta $, 
all the summands $< \omega ^\eta $ in the Cantor normal expression
of $\alpha$ are absorbed by the leading term of the normal expression of  $\beta$.

(\ref{facts3op}) By (\ref{facts5}) and associativity of $\+$, we get
 $( \alpha + \beta ) \+ \gamma =
\alpha^{\restriction  \eta}\+ \beta \+ \gamma $,
where $\eta$ is the leading exponent in the normal expression of $\beta$.
In the same way,
$\alpha + (\beta  \+ \gamma) =
\alpha^{\restriction  \xi}\+ \beta \+ \gamma $,
where $\xi$ is the leading exponent in the normal expression of $\beta \+ \gamma $.
Now trivially
$\eta \leq \xi $,
hence
$\alpha^{\restriction  \eta} \geq \alpha^{\restriction  \xi}$,
from which 
(\ref{facts3op}) follows. 

(7) By (\ref{facts5})  and associativity of $\+$, we get
$ \alpha \+ (\beta  + \gamma )
=
\alpha \+ \beta^{\restriction  \eta}  \+ \gamma
$,
where $\eta$ is the leading exponent in the normal expression of $\gamma $.
On the other hand,
$ 
(\alpha \+ \beta)  + \gamma
=
(\alpha \+ \beta)^{\restriction  \eta}  \+ \gamma
=
\alpha^{\restriction  \eta}  \+ \beta^{\restriction  \eta}  \+ \gamma
$,
thus the inequality follows from monotonicity of $\+$.
\end{proof}

\section{A transfinite iteration of the natural sum} \labbel{tins} 

\subsection*{Definition and notations} 
A countable iteration of the natural sum
has been considered in  V\"a\"an\"anen and Wang   
\cite{VW},
by taking the supremum at the limit stage.
This countable operation has
been extensively studied 
in \cite{w}. 
The construction 
can be  iterated without special adjustments
 through
the transfinite.

\begin{definition} \labbel{iterated}
If $( \alpha_ \gamma ) _{ \gamma < \bar{\varepsilon}  } $
is a sequence of ordinals, we define the
\emph{iterated natural sum}
$\sumh _{ \gamma < \delta } \alpha _ \gamma $,
for every $\delta \leq \bar{\varepsilon} $,  
inductively as follows.
 \begin{align}\labbel{trans}
\sumh _{ \gamma < 0 }  \alpha _ \gamma & =  0\\ 
\labbel{transb}
\sumh _{ \gamma < \delta+1 } \alpha _ \gamma & =
\left(\sumh _{ \gamma < \delta } \alpha _ \gamma \right)
\+
\alpha_ \delta\\
\labbel{transc}
 \sumh _{ \gamma < \delta } \alpha _ \gamma & =
\sup _{ \delta ' < \delta }\  \sumh _{ \gamma < \delta' } \alpha _ \gamma
\quad \text{ for $\delta$ limit}  
\end{align}

Clearly, as far as $\bar{\varepsilon} \geq \delta $,
the definition of  $\sumh _{ \gamma < \delta }\alpha _ \gamma $
does not depend on $\bar{\varepsilon}$.
Moreover, it does not depend on the values assumed by 
the $\alpha_ \gamma $'s, for $\gamma\geq \delta $.
This shows that there is no ambiguity in the notation. 
\end{definition}   

In the particular case when $\delta= \omega $, the operation
$\sumh _{ \gamma < \omega  } \alpha _ \gamma $ has been 
considered in \cite{VW} under the same notation and studied
in \cite{w}  under the notation
$\nsum _{ \gamma < \omega  } \alpha _ \gamma $.
We are using here a  different notation 
since we want to reserve the symbol $\nsum$
for some operation which is invariant under permutations.
In fact,  if $ \delta = \omega$, then  $\sumh$ is  invariant
while, if $ \delta > \omega$, then  $\sumh$ is not invariant.
See Section \ref{furth} for further details;
there we shall also introduce  some related 
operations which are indeed invariant under permutations.

\begin{notation} \labbel{sh} 
If the sequence of the $\alpha_ \gamma $'s 
in Definition \ref{iterated} is understood,
we shall sometimes simply write  
$\sh_ \delta $ in place of  
$\sumh _{ \gamma < \delta } \alpha _ \gamma $.
In a few cases, 
we shall also need a shorthand for 
the partial sums of the
 usual iterated ordinal sum
$\sum _{ \gamma < \delta } \alpha _ \gamma $.
This will be abbreviated as
$ \snorm _ \delta $.
I.~e.,
in the recursive definition of 
$ \snorm _ \delta $, we use
$ \snorm _{ \delta +1} = \snorm_ \delta + \alpha _\delta $
in place of equation \eqref{transb} above. 
 
As an additional shorthand, 
particularly useful to simplify notations in  proofs,
for $\delta' \leq  \delta $, 
we let 
$ \sh_{ [ \delta' , \delta )} $  denote 
$\sumh _{ \delta ' \leq  \gamma < \delta } \alpha _ \gamma $, 
and, similarly, 
$ \snorm _{ [ \delta' , \delta )} 
= \sum _{ \delta ' \leq  \gamma < \delta } \alpha _ \gamma $.
As usual, the notation
$\sumh _{ \delta ' \leq  \gamma < \delta } \alpha _ \gamma $
is justified by the fact
that, if $\delta' \leq \delta $,
then there is a unique ordinal
$ \varepsilon $ such that   
$\delta' + \varepsilon  =  \delta $,
and then the  formal definition of 
$\sumh _{ \delta ' \leq  \gamma < \delta } \alpha _ \gamma $
is
$\sumh _{ \varepsilon '   < \varepsilon  } \alpha _{ \delta ' + \varepsilon' } $.
All this is standard; see 
the mentioned books \cite{Bac,Sier}
for further details, e.~g.,  in the case of the ordinary transfinite sum.
 \end{notation}

\subsection*{Some preliminary lemmas}

Throughout the present subsection
we fix a sequence  $( \alpha_ \gamma ) _{ \gamma < \bar{\varepsilon}  } $
of ordinals,
where 
$\bar{\varepsilon}$
is some sufficiently large ordinal.

\begin{proposition} \labbel{factsww}
Let $\delta$, $\alpha_ \gamma $, $\beta_ \gamma $ 
($\gamma < \delta $)
be ordinals. Then the following statements hold. 
  \begin{enumerate}[(1)]    
\item
$\sum_{ \gamma < \delta   } \alpha _ \gamma 
\leq
\sumh _{ \gamma < \delta   } \alpha _ \gamma $
\item
If $\beta_ \gamma  \leq \alpha _ \gamma $, for every $ \gamma < \delta $,
then
$\sumh _{ \gamma < \delta   } \beta  _ \gamma \leq
 \sumh _{ \gamma < \delta   } \alpha _ \gamma$    
\item
If $ \delta' < \delta $, then
$ \sumh _{ \gamma < \delta '  } \alpha _ \gamma
 \leq
  \sumh _{ \gamma < \delta   } \alpha _ \gamma$; equality holds if and only if 
$\alpha_ \gamma = 0$,
for every $\gamma$ with $ \delta ' \leq \gamma < \delta $.   
\item
If $(\alpha _{f( \varepsilon )}) _{ \varepsilon < \delta' } $
is the subsequence of 
$(\alpha _{ \gamma }) _{ \gamma  < \delta } $
consisting exactly of the nonzero elements of the sequence, then
$  \sumh _{ \gamma < \delta   } \alpha _ \gamma =
  \sumh _{ \varepsilon < \delta ' } \alpha _{f( \varepsilon )} $. 
\item  
If $\delta'' < \delta $, $k$ is finite
and $ \delta '' \leq  \gamma _0, \gamma _1, \dots, \gamma _ {k-1} < \delta $, 
 then
$  \sumh _{ \gamma < \delta   } \alpha _ \gamma \geq
 \left( \sumh _{ \gamma  < \delta '' } \alpha _ \gamma \right) 
\+  \alpha  _{ \gamma _0} \+ \alpha  _{ \gamma _1} \+ \dots
 \+ \alpha  _{ \gamma _{k-1}}
$. 
 \end{enumerate}
 \end{proposition} 

\begin{proof}
(1)-(3) are immediate from Proposition \ref{facts}
and Definition \ref{iterated}.
(4) is proved by induction on $\delta$, using (3).

In order to prove (5),
let 
$( \alpha' _ \gamma ) _{ \gamma < \delta   } $
be the sequence obtained from $( \alpha _ \gamma ) _{ \gamma < \delta  } $,
by changing to $0$ all elements except for those elements which either
have index $< \delta '' $ or have index in the set
$\{  \gamma _0, \dots, \gamma _{k-1} \}$.
Then, by (2), and applying (4) to the sequence  
$( \alpha' _ \gamma ) _{ \gamma < \delta   } $, we get
$\sumh _{ \gamma < \delta   } \alpha   _ \gamma \geq
 \sumh _{ \gamma < \delta   } \alpha' _ \gamma =
  \sumh _{ \varepsilon < \delta ' } \alpha _{f( \varepsilon )}=
 \left( \sumh _{ \varepsilon < \delta '' } \alpha _ \gamma \right) 
\+  \alpha  _{ \gamma _0} \+ \alpha  _{ \gamma _1} \+ \dots 
\+ \alpha  _{ \gamma _{k-1}}
$, 
where $\delta'$ is given by (4),
we iterate clause (\ref{transb}) in Definition \ref{iterated}
a finite number of times, 
 and we are using the fact that
 $\delta'= \delta '' + k$.    
\end{proof}

\begin{lemma} \labbel{lem1}
Suppose that  $\delta' \leq \delta $.
Then
\begin{equation*}\labbel{min}
\sumh _{ \gamma < \delta' } \alpha_ \gamma
\+
\sumh _{ \delta' \leq \gamma < \delta } \alpha_ \gamma  
\geq
\sumh _{ \gamma < \delta } \alpha_ \gamma 
\geq
\sumh _{ \gamma < \delta' } \alpha_ \gamma
+
\sumh _{ \delta' \leq \gamma < \delta } \alpha_ \gamma  
  \end{equation*}    
 \end{lemma} 

\begin{proof}
The proof is obtained by leaving $\delta'$ fixed and performing 
an induction on
$\delta \geq \delta' $, using Proposition \ref{facts}. 
Notice that, in the shorthand introduced in  \ref{sh},
the conclusion of the lemma reads 
$\sh _{\delta' } 
\+
\sh _{[ \delta', \delta) }
\geq
\sh _{  \delta } 
\geq
\sh _{\delta' } 
+
\sh _{[ \delta', \delta) }$.

We first prove the left-hand inequality.
If $\delta= \delta' $, then
 $\sh _{[ \delta', \delta) }$ is an empty sum with value $0$
and the result is trivial.

If $\delta$  is successor, say, $\delta= \delta'' +1 $, 
then, by the inductive hypothesis,
$\sh _{\delta' } 
\+
\sh _{[ \delta', \delta'') }
\geq
\sh _{  \delta'' }$.
Then, by definition,
$\sh _{\delta' } 
\+
\sh _{[ \delta', \delta) }
=
\sh _{\delta' } 
\+
(\sh _{[ \delta', \delta'') }
 \+ \alpha _{ \delta ''})
=
(\sh _{\delta' } 
\+
\sh _{[ \delta', \delta'') })
 \+ \alpha _{ \delta ''}
\geq
\sh _{  \delta'' }  \+ \alpha _{ \delta ''} 
=
\sh _{  \delta }$,
by associativity \ref{facts}(1).

If $\delta$ is limit, 
then, by Proposition \ref{factsww}(3)  and  the inductive hypothesis,
 $\sh _{  \delta' } 
\+
\sh _{ [\delta' , \delta)} 
\geq
\sh _{  \delta' } 
\+
\sh _{ [\delta' , \delta^* )}  
\geq 
\sh _{ \delta^* }$,
for every $\delta^*$ with  $ \delta ' \leq \delta^* < \delta $. 
Hence
 $\sh _{  \delta' } 
\+
\sh _{ [\delta' , \delta)}
\geq
\sup _{ \delta ^* < \delta } \sh _{ \delta^* }
=
\sh _{  \delta} $.

Now let us prove the right-hand inequality.
Again, this is trivial if $\delta= \delta '$.

If $\delta$  is successor, say, $\delta= \delta'' +1 $, 
then, by the inductive hypothesis,
$\sh _{  \delta'' } \geq
\sh _{\delta' } 
+
\sh _{[ \delta', \delta'') }$.
Then, by definition and using Proposition \ref{facts}(\ref{facts3op}), 
$\sh _{  \delta } 
=
\sh _{  \delta'' }  \+ \alpha _{ \delta ''} 
\geq
(\sh _{\delta' } 
+
\sh _{[ \delta', \delta'') })
 \+ \alpha _{ \delta ''}
\geq
\sh _{\delta' } 
+
(\sh _{[ \delta', \delta'') }
 \+ \alpha _{ \delta ''})
=
\sh _{\delta' } 
+
\sh _{[ \delta', \delta) }
$.

If $\delta$ is limit, 
then, by Proposition \ref{factsww}(3) and the inductive hypothesis,
 $\sh _{ \delta} \geq \sh _{ \delta^* }  \geq
\sh _{  \delta' } 
+
\sh _{ [\delta' , \delta^* )}  $,
for every $\delta^*$ with  $ \delta ' \leq \delta^* < \delta $. 
Hence
 $\sh _{ \delta} \geq 
\sup _{ \delta^* < \delta } (\sh _{  \delta' } 
+
 \sh _{ [\delta' , \delta^* )} )
=
\sh _{  \delta' } 
+
\sup _{ \delta^* < \delta } \sh _{ [\delta' , \delta^* )} 
=
\sh _{  \delta' } 
+
 \sh _{ [\delta' , \delta)} $,
 by right continuity of $+$.
 \end{proof}   
 
\begin{remark} \labbel{rmklem1}
Notice that the inequalities  in Lemma \ref{lem1}
might be strict.
For example, taking $\delta= \omega $, $\delta'=1$
and $\alpha_ \gamma =1$, for every $\gamma < \omega $,
we get
$
1 
\+ 
\sumh _{ 1 \leq \gamma < \delta } \alpha_ \gamma
= \omega +1 >  \omega  =
\sumh _{ \gamma < \delta } \alpha_ \gamma
  $, showing that the left-hand inequality might be strict.

On the other hand, take 
$\delta= 2$,
$\delta'=1$,
$\alpha_0= 1$   
and 
$\alpha_1= \omega $,
getting
$
\sumh _{ \gamma < \delta } \alpha_ \gamma
=1 
\+ \omega 
= \omega +1 >  \omega  = 1+ \omega =
1+
\sumh _{ 1 \leq \gamma < \delta } \alpha_ \gamma$.
 \end{remark}

\begin{lemma} \labbel{lm}
If $ n < \omega $ and
$\sumh _{ \gamma < \delta } \alpha_ \gamma \geq \omega ^ \xi n $,
then
$\sum _{ \gamma < \delta } \alpha_ \gamma \geq \omega ^ \xi n $.  
 \end{lemma} 

\begin{proof}
By induction on ordinals $\iota$
of the form 
$\omega ^ \xi n$.

The lemma is trivial if $\iota$ is either $0$ or  $1$. 
 
Suppose that 
$\iota=\omega ^ \xi n>1$,
and that the lemma is true
for all $\iota' < \iota$
of the  form $\omega ^{ \xi'} n'$.
Suppose by contradiction that the lemma is  false for 
$\iota$, and choose a counterexample for which $ \delta $ 
is minimal.
Recalling the notation from
\ref{sh}, 
 we are assuming that
$\sh _{  \delta } \geq \omega ^ \xi n $,
but
$\snorm _{  \delta }  < \omega ^ \xi n $,
and that $\delta$ is minimal with the above properties. 
Then  
$\sh _{  \delta' }  < \omega ^ \xi n $,
for every $\delta' < \delta $, since,
otherwise, $\sh _{ \delta' }  \geq  \omega ^ \xi n $,
and
either $\delta'$ would give a counterexample,
contradicting the minimality of $\delta$,
or 
$\snorm _{  \delta}  \geq
\snorm _{  \delta' } \geq  \omega ^ \xi n $,
contradicting the assumption that
$\delta$ gives a counterexample.

Towards a contradiction, we shall
exclude all the possibilities for $\delta$.
Necessarily, $\delta>0$, since $\iota >0$.

We shall now exclude the case when $\delta$  is successor.
Suppose that $\delta= \varepsilon + 1$.
Since, by definition,
$ \sh _{  \varepsilon  } \+ \alpha_ \varepsilon   =
\sh _{ \delta   }  \geq \omega ^ \xi n$,
we get by Proposition \ref{facts}(\ref{factsn1n2})
that there are $n_1$ and $n_2$
such that $n_1+n_2=n$ and   
$\sh _{ \varepsilon  }  \geq  \omega ^ \xi n_1$,
$\alpha_ \varepsilon  \geq  \omega ^ \xi n_2$.
Since we know that
 $\sh _{ \delta' }  < \omega ^ \xi n $,
for every $\delta' < \delta $,
we 
have  $\sh _{ \varepsilon }  < \omega ^ \xi n $, so
that  $n_1< n$,
hence we can apply the inductive hypothesis with
$\iota' = \omega ^ \xi n_1$,
getting 
$\snorm _{ \varepsilon  } \geq  \omega ^ \xi n_1$.
Then 
$\snorm _{ \delta   }  =
 \snorm _{ \varepsilon  }  + \alpha_ \varepsilon  
\geq
 \omega ^ \xi n_1+ \omega ^ \xi n_2 =
\omega ^ \xi n$, a contradiction.
Thus we cannot get a counterexample at $\delta$ successor.
 
The remaining case is when $\delta$ is limit.
The argument shall be split into two cases.
First assume that $n=1$, that is,  
$\iota=\omega ^ \xi > 1$.
Then $\iota = \sup _{\iota ' < \iota} \iota'$, with, as above, 
the $\iota'$s varying on ordinals of the form   
$\omega ^{ \xi'} n'$.
Since 
$\sh _{  \delta}  \geq \omega ^ \xi  > \iota'$,
for every $\iota ' < \iota$,
we can apply the inductive hypothesis,
getting
 $\snorm _{ \delta}    \geq  \iota'$,
for every $\iota ' < \iota$,
hence
 $\snorm _{  \delta} \geq \sup _{\iota ' < \iota} 
  \iota' = \iota= \omega ^ \xi$,
contradicting the assumption that $\delta$ gives a counterexample to the
statement of the lemma.

The only case left is when $\delta$ is limit and 
$\iota=\omega ^ \xi n$,
for some $n>1$. 
We know that
$\sh _{  \delta' } < \omega ^ \xi n $,
for every $\delta' < \delta $.
But also
$ \sup _{ \delta' < \delta } \sh _{  \delta' }  
=
\sh _{  \delta }  \geq\omega ^ \xi n $,
hence there is some $\delta' < \delta $
such that  
$\sh _{  \delta' }  \geq \omega ^ \xi (n-1)$.
Fix some $\delta'$ satisfying the above inequality. 
By Lemma \ref{lem1},
$ \sh _{ \delta' } 
\+
\sh _{ [\delta', \delta) } 
\geq
\sh _{ \delta } $.
Since 
$\sh _{ \delta } \geq \omega ^ \xi n$
 and 
$ \sh _{ \delta' }  < \omega ^ \xi n$,
we necessarily get
$\sh _{ [\delta', \delta) }  \geq \omega ^\xi $,
by Proposition \ref{facts}(\ref{factsn1n2}). 
We can now apply the inductive hypothesis twice, getting 
$\snorm _{  \delta' }   \geq \omega ^ \xi (n-1)$
from
$\sh _{  \delta' }  \geq \omega ^ \xi (n-1)$
and
 $\snorm _{ [\delta' , \delta) }  \geq \omega ^\xi $
from
 $\sh _{ [\delta' , \delta) }  \geq \omega ^\xi $,
since $n >1$, hence
 $  \omega ^ \xi <\omega ^ \xi n =\iota $.
Then
$
\snorm _{ \delta}  =
\snorm _{  \delta' }   
+
\snorm _{ [\delta' , \delta )} 
 \geq \omega ^\xi (n-1) + \omega ^\xi =
\omega ^\xi n$.
Again, this contradicts the assumption that $\delta$
gives a counterexample.

We have showed that the lemma holds for 
$\iota$, assuming that it holds for every 
$\iota' < \iota$ of the form $\omega ^\xi n$.
By induction, the lemma holds for every 
$\iota$ of the form   
$\omega ^\xi n$.
 \end{proof}

\subsection*{The sums differ only for a finite number of inductive steps}

If $\alpha>0$ is an ordinal
expressed in Cantor normal form as
 $\alpha  =
 \omega ^ {\xi_k} r_k + 
\dots
+ \omega ^ {\xi_0}r_0  $ 
(hence $r_0 > 0$),
we call $\xi_0$
the \emph{smallest exponent}
of $\alpha$.

\begin{theorem} \labbel{infsum}
Suppose that $\zeta$  is a limit ordinal,
$(\alpha_ \gamma ) _{ \gamma < \zeta } $
is a  sequence of ordinals
which are not eventually zero,
and let $\xi$ be the smallest exponent of (the non zero ordinal)
$\sumh _{\gamma < \zeta } \alpha _ \gamma $.

Then there is $\bar{ \gamma } < \zeta $ such that,
for every $ \varepsilon \geq \bar{ \gamma }$  
\begin{align}\labbel{eqsums} 
\sumh _{ \varepsilon  \leq \gamma < \zeta } \alpha _ \gamma & = 
\sum _{ \varepsilon  \leq \gamma < \zeta } \alpha _{ \gamma}
= \omega ^ \xi  \quad \text{ and } \\ 
 \labbel{eqsumss}
\sumh _{\gamma < \zeta } \alpha _ \gamma & =
\sumh _{\gamma < \varepsilon } \alpha _ \gamma 
+ \omega ^ \xi
\end{align}      
 \end{theorem}  

\begin{proof} 
Let  $ \sh _{ \zeta }= \sumh _{\gamma < \zeta } \alpha _ \gamma 
 =
 \omega ^ {\xi_k} r_k + 
\dots
+ \omega ^ {\xi}r $ 
be expressed in Cantor normal form.
Let $ \varepsilon  < \zeta$.
Since the sequence
$( \alpha _ \gamma ) _{\gamma < \zeta} $ 
is  not eventually zero, then
$\sh _{ \varepsilon  } 
<
\sh _{\zeta }  
$,
by Proposition \ref{factsww}(3), 
hence $\sh _{ \varepsilon  } + \omega ^ \xi
\leq
\sh _{\zeta }$, 
expressing $\sh _{ \varepsilon }$ in Cantor normal form and 
using Proposition \ref{facts}(\ref{facts5}).
Since, by Lemma \ref{lem1},
$\sh _{ \zeta } \leq
 \sh _{ \varepsilon }\+
 \sh _{ [\varepsilon , \zeta )}$, we get
$\sh _{ [\varepsilon , \zeta )} \geq \omega ^ \xi$,
by Proposition \ref{facts}(\ref{factsex4})
with $\alpha= \sh _{ \varepsilon  }$
and $\beta= \sh _{ [\varepsilon , \zeta )}$
(were $\sh _{ [\varepsilon , \zeta )} < \omega ^ \xi$,
we would get $\sh _{ \zeta } \leq
 \sh _{ \varepsilon }\+
 \sh _{ [\varepsilon , \zeta )}
<
 \sh _{ \varepsilon } + \omega ^\xi \leq \sh _{ \zeta }$,
a contradiction).
By Lemma \ref{lm},
also   
$\snorm _{ [\varepsilon , \zeta )} \geq \omega ^ \xi$.
All this holds for every  $ \varepsilon  < \zeta$.

On the other hand,
since $\zeta$ is limit, thus
$
\sh _{ \zeta }  =
\sup _{ \varepsilon < \zeta } \sh _{ \varepsilon  } $,
and since $r>0$, 
then there is  some 
$\bar{ \gamma } < \zeta $
such that  
 $ \sh _{ \bar{ \gamma }}\geq 
 \omega ^ {\xi_k} r_k + 
\dots
+ \omega ^ {\xi}(r-1) $.
 Then 
 $ \sh _{ \varepsilon }\geq 
 \sh _{ \bar{ \gamma }}\geq 
 \omega ^ {\xi_k} r_k + 
\dots
+ \omega ^ {\xi}(r-1) $,
for every $\varepsilon$ such that 
$\zeta  >
 \varepsilon \geq 
\bar{ \gamma }$.
By Lemma \ref{lem1} 
and the previous paragraph we get that if
$\zeta  >
 \varepsilon \geq 
\bar{ \gamma }$, then
$ \omega ^ {\xi_k} r_k + 
\dots
+ \omega ^ {\xi}r =
\sh _{ \zeta } \geq 
 \sh _{ \varepsilon }+
 \sh _{ [\varepsilon , \zeta )}
\geq 
\omega ^ {\xi_k} r_k + 
\dots
+ \omega ^ {\xi}(r-1) +
 \sh _{ [\varepsilon , \zeta )}
\geq 
\omega ^ {\xi_k} r_k + 
\dots
+ \omega ^ {\xi}r$,
hence all are equal, and we obtain at once
$\sh _{ [\varepsilon , \zeta )} = \omega ^ {\xi}$
(e.~g., by Proposition \ref{facts}(\ref{facts5})),
as well as 
\eqref{eqsumss}.
Hence, recalling the above paragraph, we get also
 $ \omega ^ \xi \leq \snorm _{ [\varepsilon , \zeta )} \leq 
\sh _{ [\varepsilon , \zeta )} = \omega ^ {\xi}$,
thus the proof of 
\eqref{eqsums} is complete.
\end{proof}

We can introduce 
 intermediate variants between
$\sum _{\gamma < \zeta } \alpha _ \gamma $
and
$\sumh _{\gamma < \zeta } \alpha _ \gamma $
by taking the natural sum at certain successor stages
and the usual ordinal sum at the remaining successor stages.
As a simple consequence of Theorem \ref{infsum},
we will get in Corollary \ref{finstep} below  that 
$\sumh _{\gamma < \zeta } \alpha _ \gamma $
can be computed by taking natural sums at just a
finite number of stages.
In more details, we are dealing with the notion introduced in the following
definition.

\begin{definition} \labbel{stages}
Suppose that $\zeta$ is an ordinal, 
$(\alpha_ \gamma ) _{ \gamma < \zeta } $
is a  sequence of ordinals, and $G$ is a set of ordinals.
For every $\delta \leq \zeta  $,  we define the
\emph{partial natural sum (relative to  $G$)}
$\sumg _{ \gamma < \delta } \alpha _ \gamma $
inductively as follows.
 \begin{align*}
 \sumg _{ \gamma < 0 }   \alpha _ \gamma & =  0\\ 
 \sumg _{ \gamma < \delta+1 }  \alpha _ \gamma &  =
\left(\sumg _{ \gamma < \delta } \alpha _ \gamma \right)
\+
\alpha_ \delta \quad \text{ if $\delta \in G$} \\
 \sumg _{ \gamma < \delta+1 } \alpha _ \gamma & =
\left(\sumg _{ \gamma < \delta } \alpha _ \gamma \right)
+
\alpha_ \delta \quad \text{ if $\delta \not\in G$} \\
 \sumg _{ \gamma < \delta } \alpha _ \gamma &  =
\sup _{ \delta ' < \delta }\  \sumg _{ \gamma < \delta' } \alpha _ \gamma
\quad \text{ for $\delta$ limit}  
\end{align*}

Notice that, in particular, if $G \supseteq  \delta  $,
then 
$\sumg _{ \gamma < \delta } \alpha _ \gamma 
= 
\sumh _{ \gamma < \delta } \alpha _ \gamma $,
while if 
 $G \cap \delta = \emptyset $,
then 
$\sumg _{ \gamma < \delta } \alpha _ \gamma 
= 
\sum _{ \gamma < \delta } \alpha _ \gamma $.

Of course, the definition of 
$\sumg _{ \gamma < \delta } \alpha _ \gamma $
depends only on
$G \cap \delta $, that is, we could have assumed that
$G \subseteq \zeta $.
However, since we shall
have frequent occasion to deal with partial sums,  it will be notationally convenient
to take $G$ as an arbitrary set of ordinals.
 \end{definition}   

As custom by now, let 
$S^G_ \varepsilon $ abbreviate
$\sumg _{ \gamma < \varepsilon  } \alpha _ \gamma $,
and similarly when considering other sets of ordinals in place of $G$.
Also
$S ^{G } _{[ \delta , \zeta  )}$ has the usual meaning, that is,
$S ^{G } _{[ \delta , \zeta  )} = 
\sumg _{ \delta \leq  \gamma < \zeta  } \alpha _ \gamma =
\sumhh _{ \varepsilon '   < \varepsilon  } \alpha _{ \delta  + \varepsilon' } $,
where 
$\varepsilon$ is the only ordinal such that 
$\delta+ \varepsilon = \zeta $, 
and $ \varepsilon ' \in H$
if and only if  
$ \delta + \varepsilon ' \in G$. 

\begin{lemma} \labbel{ginter}
Under the assumptions in Definition \ref{stages},  if   $\delta \leq \zeta $, then 
\begin{equation*}\labbel{gs}
\sumg _{  \gamma < \delta   } \alpha _ \gamma 
\+
\sumg _{ \delta \leq  \gamma < \zeta  } \alpha _ \gamma 
\geq 
\sumg _{  \gamma < \zeta  } \alpha _ \gamma 
\geq
\sumg _{  \gamma < \delta   } \alpha _ \gamma 
+
\sumg _{ \delta \leq  \gamma < \zeta  } \alpha _ \gamma 
 \end{equation*}    
 \end{lemma}

 \begin{proof} 
Similar to the proof of
Lemma \ref{lem1}.
At certain points in the successor induction step
one needs Proposition \ref{facts}(7). 
\end{proof}    

As a consequence of Theorem \ref{infsum}, we get  a way for
computing  
$\sumg _{ \gamma < \zeta } \alpha _ \gamma $
in the case when $\zeta$ is limit.

\begin{corollary} \labbel{computeg}
Under the assumptions in  Theorem \ref{infsum}
and if $G \subseteq \zeta $, 
 there is $\bar{ \gamma } < \zeta $ such that,
for every $ \varepsilon \geq \bar{ \gamma }$,  
\begin{equation*} 
\sumg _{\gamma < \zeta } \alpha _ \gamma =
\sumg _{\gamma < \varepsilon } \alpha _ \gamma 
+ \omega ^ \xi
\end{equation*}      
 \end{corollary}  

\begin{proof}
Let $\bar{\gamma}$ be given by
Theorem  \ref{infsum}. 
First,
using 
the case $ \varepsilon = \bar{\gamma}$ in
equation \eqref{eqsums} 
in Theorem \ref{infsum},
compute
$\omega ^\xi =
\sh _{[\bar{\gamma}, \zeta  )} \geq
S ^{G } _{[\bar{\gamma}, \zeta )} \geq
\snorm _{[\bar{\gamma}, \zeta  )}
=\omega ^\xi$,
hence
$S ^{G } _{[\bar{\gamma}, \zeta  )}
=\omega ^\xi$.
Using Lemma \ref{ginter}, we get
$S ^{G } _{\zeta }
\geq 
S ^{G } _{\bar{\gamma}} + S ^{G } _{[\bar{\gamma}, \zeta )} =
S ^{G } _{\bar{\gamma}} + \omega ^\xi$.
To prove the reverse inequality,
notice that, since 
$(\alpha_ \gamma ) _{ \gamma < \zeta } $ is not eventually zero,
we have that 
$S ^{G } _{[\bar{\gamma}, \varepsilon )} <
S ^{G } _{[\bar{\gamma}, \zeta  )}
=\omega ^\xi$,
for every $\varepsilon$ such that 
$\bar{\gamma} \leq \varepsilon < \zeta  $. 
Then $ S ^{G } _{ \varepsilon  } 
\leq
S ^{G } _{\bar{\gamma}} \+ S ^{G } _{[\bar{\gamma}, \varepsilon  )}
<
 S ^{G } _{\bar{\gamma}} + \omega ^\xi$,
for  $\varepsilon$ with
$\bar{\gamma} \leq \varepsilon < \zeta  $, 
by Lemma \ref{ginter} and Proposition \ref{facts}(\ref{factsex4}). 
Hence 
$ S ^{G } _{ \zeta   } 
= \sup _{ \varepsilon < \zeta  } 
 S ^{G } _{ \varepsilon  } 
\leq
 S ^{G } _{\bar{\gamma}} + \omega ^\xi$.
\end{proof}

\begin{corollary} \labbel{finstep}
Suppose that $\zeta$ is an ordinal and 
$(\alpha_ \gamma ) _{ \gamma < \zeta } $
is a  sequence of ordinals. Then the following hold.
  \begin{enumerate}[(1)]    
\item   
There is a finite set $F  $
(depending on the sequence)
such that  $F \subseteq \zeta $  and
$\sumh _{ \gamma < \zeta  } \alpha _ \gamma 
= 
\sumf _{ \gamma < \zeta } \alpha _ \gamma $.
\item
More generally, we can choose a finite $F \subseteq \zeta $ in such a way that
$\sumg _{ \gamma < \zeta  } \alpha _ \gamma 
= 
\sumhh _{ \gamma < \zeta  } \alpha _ \gamma $,
whenever 
$G, H  $ are sets of ordinals such that 
 $G \cap F = H \cap F$.
\item
In particular, for every set $G$ of ordinals, 
there is some finite $H \subseteq \zeta $
such that 
$\sumg _{ \gamma < \zeta  } \alpha _ \gamma 
= 
\sumhh _{ \gamma < \zeta  } \alpha _ \gamma $.
Letting $G$ vary among all subsets of $\zeta$,
we get only a finite number of values for
$\sumg _{ \gamma < \zeta } \alpha _ \gamma $.
\end{enumerate} 
\end{corollary}

 \begin{proof}
We first notice that (3) follow
from (2): just take $H= G \cap F$,
where $F$ is given by (2). 
Moreover, the set 
$\{ G \cap F \mid G \subseteq \zeta  \}$ is finite. 

We shall prove (1) and (2) by induction on $\zeta$.

The result is trivial when $\zeta= 0$.
The step from $\zeta$ to $\zeta+1$ is trivial, too:
if $F$ works for $\zeta$, then 
surely $F \cup \{ \zeta \} $ works for  $\zeta+1$.

Hence let us assume that $\zeta$ is limit.  
If the sequence 
$(\alpha_ \gamma ) _{ \gamma < \zeta } $ is constantly
zero from some point on, 
say 
$\alpha_ \gamma =0$ for $\gamma \geq \zeta '$,
then we can apply the inductive hypothesis 
for $\zeta'$, getting some $F$ 
working for  
  $(\alpha_ \gamma ) _{ \gamma < \zeta '} $.
But then trivially $F$ 
works for the original sequence $(\alpha_ \gamma ) _{ \gamma < \zeta } $,
too.
Hence we can suppose that
$(\alpha_ \gamma ) _{ \gamma < \zeta } $ is not eventually zero
and apply Theorem \ref{infsum},
getting some
$\bar{\gamma}$
for which equations \eqref{eqsums} and
\eqref{eqsumss} there holds.

Now (1) is easy.
By the inductive hypothesis,
there is some finite $F \subseteq \bar{\gamma} $  such that
$\sh _{ \bar{\gamma}}
= 
S^F _{ \bar{\gamma}}$.
Then equations \eqref{eqsumss} and \eqref{eqsums} 
with $ \varepsilon = \bar{\gamma}$ 
 give
$\sh _{ \zeta }
= 
\sh _{  \bar{\gamma}} + \omega ^ \xi
=
S^F _{ \bar{\gamma}} + \snorm _{[\bar{\gamma}, \zeta )} 
=
S^F _{ \zeta } $,
where the last identity is proved
by induction 
on $\zeta'$,
with $\bar{\gamma} \leq \zeta' \leq  \zeta$,
and where the successor step
uses the assumption  
$F \subseteq    \bar{\gamma}$,
and the limit step uses (right) continuity of $+$
at limits. 

To prove (2), use again  the inductive hypothesis
to get some finite $F \subseteq \bar{\gamma} $  such that
$ S^G _{\bar{\gamma}} 
= 
S^H _{\bar{\gamma}} $,
whenever 
 $G \cap F = H \cap F$.
We claim that $F$ works for
 $\zeta$, too, that is, 
$ S^G _{ \zeta } 
= 
S^H _{ \zeta } $,
whenever 
$G \cap F = H \cap F$.
Indeed, by Corollary \ref{computeg}, 
$S^G _{ \zeta }=
S ^{G } _{\bar{\gamma}} + \omega ^\xi $,
for every $G \subseteq \zeta $
(notice that in Corollary \ref{computeg}
we can choose the same 
$\bar{\gamma}$ as the one given by Theorem \ref{infsum}).  
This is enough, since if
$G \cap F = H \cap F$, then,
by the inductive assumption
$ S^{G }_{\bar{\gamma}} 
= 
S ^{H } _{\bar{\gamma}} $,
and then
$S^G _{ \zeta }=
S ^{G } _{\bar{\gamma}} + \omega ^\xi =
S ^{H} _{\bar{\gamma}} + \omega ^\xi =
S^H _{ \zeta }
$.
 \end{proof}

\section{Various kinds of mixed sums} \labbel{mixedsums}

\begin{definition} \labbel{infmix}
If $(\alpha_i) _{i \in I} $
is a sequence of ordinals, we say that
an ordinal $ \beta $ is a \emph{mixed sum}
of $(\alpha_i) _{i \in I} $  if there are 
pairwise disjoint
subsets $(A_i) _{i \in I} $ of $ \beta $
such that 
$\bigcup _{i \in I} A_i = \beta  $
and, for every $i \in I$,
$A_i$ has order type $\alpha_i$,
with respect  to the order induced on $A_i$ 
by $ \beta $.    

In the above situation, we say
that $ \beta $ is a mixed sum of
$(\alpha_i) _{i \in I} $
\emph{realized by 
$(A_i) _{i \in I} $},
or simply that 
$ \beta $ is realized by 
$(A_i) _{i \in I} $
(notice that the $\alpha_i$'s can be retrieved from
the $A_i$'s, hence the terminology is not ambiguous).

Given a realization
$(A_i) _{i \in I} $ and 
$\bar{\imath} \in I$,
we say that 
$A _{\bar{\imath}} $
is \emph{convex (in the realization $(A_i) _{i \in I} $)} if 
$(a,a') _ \beta = \{ \beta ' < \beta \mid a < \beta ' < a' \} \subseteq A _{\bar{\imath}}  $,
for every $a < a' \in A _{\bar{\imath}} $.  
 \end{definition}   

Carruth \cite{Car}, in different terminology,  
showed that $\alpha_ 1 \+ \alpha _2 $
is the largest mixed sum 
of $\alpha_1$ and $ \alpha _2$. 
In general,
when $I$ is infinite,
there is no largest mixed sum of a sequence of ordinals;
see the comment after Theorem 4.2  in  \cite{w}.
Hence Carruth theorem can be generalized only
if one restricts to mixed sums satisfying particular properties.
This has been done in \cite{w} in the case of 
mixed sums related to
the countably infinite natural sum.
We shall treat here the case of 
arbitrary iterated natural sums,
a case which is slightly more involved.

\begin{theorem} \labbel{mix}
Suppose that $\zeta$ is an ordinal,
$(\alpha_ \gamma ) _{ \gamma < \zeta } $
is a  sequence of ordinals and 
$G \subseteq \zeta $. 
Put
$\beta= \sumg _{ \gamma < \zeta} \alpha _ \gamma  $. 
Then $\beta$ is a mixed sum of $(\alpha_ \gamma ) _{ \gamma < \zeta } $.
Moreover, $\beta$ can be realized by 
$(A_ \gamma ) _{ \gamma < \zeta } $
in such a way that
  \begin{enumerate}[(1)]    
\item 
For every $\varepsilon < \zeta $,
the set $ \Gamma_ \varepsilon =
  \{ \gamma < \zeta  \mid \gamma > \varepsilon 
\text{ and } b< a, \text{ for some } b \in A_ \gamma   
\text{ and } a \in A_ \varepsilon  \}$  is finite, and
\item 
all but a finite number of 
the $A_ \varepsilon$'s
are convex in the realization. 
  \end{enumerate}
 \end{theorem}

\begin{corollary} \labbel{mixx}
In particular, $ \sumh _{ \gamma < \zeta} \alpha _ \gamma  $
is a mixed sum of  
$(\alpha_ \gamma ) _{ \gamma < \zeta } $
and it can be realized in such a way that (1) and (2)
above hold.
Moreover, $ \sumh _{ \gamma < \zeta} \alpha _ \gamma  $  is the largest 
 mixed sum of $(\alpha_ \gamma ) _{ \gamma < \zeta } $
that can be realized in such a way that the following condition
(weaker than (1) above) is satisfied.
  \begin{enumerate}[(1)]    
\item [(3)]
For every $\varepsilon < \zeta $ and
$a \in A_ \varepsilon $, the set
$ \{ \gamma < \zeta  \mid \gamma > \varepsilon 
\text{ and } b< a, \text{ for some } b \in A_ \gamma   \}$  is finite.  
\end{enumerate}

In particular, 
$ \sumh _{ \gamma < \zeta} \alpha _ \gamma  $ is also
 the largest 
 mixed sum of $(\alpha_ \gamma ) _{ \gamma < \zeta } $
that can be realized in such a way that both (1) and (2) hold.
 \end{corollary}

Before we can prove Theorem  \ref{mix} and Corollary \ref{mixx} 
we need to introduce some auxiliary  definitions.
These will be needed in order to state a stronger version
of condition (2) above. The stronger condition will
make the inductive proof easier. 

\begin{definition} \labbel{blocks}
Suppose that $ \delta  > 0$ is an ordinal.
 Express $ \delta $ in Cantor normal form as
$ \delta   =
 \omega ^ {\xi_k} r_k + \omega ^ {\xi _{k-1}}r _{k-1}    +
\dots
+ \omega ^ {\xi_1} r_1 + \omega ^ {\xi_0}r_0 $. 
We say that a subset $B$ of $ \delta $ 
is a \emph{block of $ \delta $} if
$B $ has the form
$[ \beta _1, \beta _2) =
\{ \beta' \mid \beta_1 \leq \beta' < \beta_2  \} $,
 for some $\beta_1$ and
$\beta_2$ 
of the form  $\beta_1= \omega ^ {\xi_k} r_k +  \dots +\omega ^ {\xi _{\ell}}s $
and $ \beta _2 = 
\omega ^ {\xi_k} r_k +  \dots +\omega ^ {\xi _{\ell}}(s+1)$,
where    $\omega ^ {\xi _{\ell}}$
actually appears in the normal expression 
of $ \delta $ and $s+1 \leq r_ \ell$
(we allow $s=0$, and we allow $\ell= k$,
thus $\beta_1$ is allowed to be $0$).
Notice that a block has order type
$ \omega ^ \xi$, for some $\xi$ 
(the block in the previous statement
has order type  $ \omega ^{\xi _ \ell}$).
  
Essentially, in the above terminology, the Cantor normal form
of $ \delta $ provides the way of realizing $ \delta $  as a finite
sequence of blocks, one put after the other in decreasing order 
with respect to length
(here and below,  \emph{decreasing} 
is intended in the broader sense, not necessarily in the sense
of \emph{strictly decreasing}).
 \end{definition}   

\begin{definition} \labbel{natsumbl} 
Suppose that $\beta$ is a mixed sum 
of $\alpha_1$ and $\alpha_2$, 
realized by 
$A_1$, $A_2$.
By a slight abuse of terminology,
we call $B \subseteq  A_1$ 
a \emph{block of $A_1$}
if $B$ is the image of some block of $\alpha_1$
under the order preserving bijection
(recall that, by definition, $A_1$ has order type $\alpha_1$),
and the same for $B \subseteq A_2$.
Notice that a block of $A_1$
is not necessarily also a block of $\beta$ 
(see below for details). 

We say that the realization
$A_1$, $A_2$ is \emph{pure}
 if, given a block $B_1$ of $A_1$
and a block $B_2$ of $A_2$,
either all elements of $B_1$ precede
all elements of $B_2$, or conversely.
In other words, a realization is pure
if all blocks from $A_1$ and all blocks 
from $A_2$ are convex subsets of   $ \beta =A_1 \cup A_2$.

We say that $\beta$ is a \emph{pure mixed sum} 
of $\alpha_1$ and $\alpha_2$
if it has some pure realization.
\end{definition}

What will be relevant here is that if
$\beta =\alpha_1 \+\alpha_2$,
then $\beta$ has a pure realization.
This is the standard way to show
that $\alpha_1 \+\alpha_2$ is 
 a mixed sum 
of $\alpha_1$ and $\alpha_2$:
just take all the blocks from both $\alpha_1 $ and $ \alpha_2$
and ``put them together'' ordered by decreasing length.
In this case, the blocks of $\beta$ are exactly
the (images of the) blocks of  $\alpha_1$ and of
$\alpha_2$.  

Notice that not every mixed sum is pure.
For example, $ \omega$ is a mixed sum of
$\alpha_1 = \omega $ and $ \alpha_2 = \omega $,
but the only pure mixed sum of 
$\alpha_1 $ and $ \alpha_2$ is $ \omega + \omega $.
Notice also that, 
for every $\alpha_1$ and $\alpha_2$,  
both $\alpha_1 +\alpha_2$
and
$\alpha_2 +\alpha_1$
are pure mixed sums of
$\alpha_1 $ \and $ \alpha_2$.
This shows that if $A_1$, $A_2$ is a pure realization 
of some $\beta$,   it is not necessarily the case that 
every block of $A_1$ (or of $A_2$) is also a block  of $\beta$.
For example, if $\alpha_1=1$, $\alpha_ 2 =  \omega $ and
$\beta= \alpha _1 + \alpha _2 = \omega $, we get
the only pure realization of $\beta$ by putting all the elements
of $\alpha_2$ after the element of $\alpha_1$.
In this example, $\beta$ is a single block, which is the union of a 
block from $\alpha_1$ and a block from  $\alpha_2$.      

On the other hand,  we do have that, in 
a pure realization, each block of the components is
a \emph{subset} of some block of the mixed sum.

\begin{lemma} \labbel{blockrep}
Suppose that $A_1$, $A_2$ is a pure realization of $\beta$.
Then each block of $A_1$ is contained in some block
of $\beta$, and similarly each block of $A_2$ is contained in some block
of $\beta$.
 \end{lemma} 

\begin{proof}
Indeed, suppose that $B_1$ 
 is a block of $A_1$ and, say,
$B_1$ has
type $ \omega^ \xi $.
There are two cases.
Either $B_1$ precedes some block from $A_2$
of type $ \omega ^ \eta$, for some $ \eta > \xi$,
or all blocks from $A_2$ which are after $B_1$ have type
$\leq  \omega^ \xi$.
We can speak of ordering between blocks 
since the realization is pure.

In the first case, let $B_2$ be the first block from $A_2$ 
which lies after $B_1$.
Notice that, by definition, all blocks of $A_2$ are disposed
in decreasing length, hence if some block
of $A_2$ after $B_1$ is strictly longer than $B_1$,
this is also true for the first block   
of $A_2$ after $B_1$.
Hence $B_2$ has type, say, $ \omega ^ \eta > \omega ^ \xi$.
Then $B_1$ is 
 ``absorbed''  by $B_2$,
that is, they lie in the same block of $\beta$
(if there are further blocks of $A_1$ between $B_1$ and
$B_2$, they are either as long as or shorter than $B_1$, hence they, too,
are absorbed by $B_2$).

By the same reason, in the second case, $B_1$ is contained in a block of $\beta$,
since all blocks after $B_1$, either from $A_1$ or from $A_2$
are shorter than or as long as $B_1$
($B_1$ might absorb some block of $A_2$ which
lies before it, hence $B_1$ is not necessarily a block of $\beta$).   
 
A symmetrical argument works for each block of $A_2$, thus the lemma is 
proved. 
 \end{proof}

Probably the notion of a pure mixed sum
(both in the case of a finite number
and of an infinite number of summands)
deserves further study, but we shall not pursue it here. 

\begin{lemma} \labbel{mixlem}
Theorem  \ref{mix} holds
when clause (2) there is strengthened to 
  \begin{enumerate}[(1)]   
 \item[(2$\,'$)]
There is some finite $F \subseteq \zeta $
such that  
if $\varepsilon \in \zeta  \setminus F$, 
then  $A_ \varepsilon$ is
convex and is contained in some block of $\beta$. 
 \end{enumerate} 
 \end{lemma} 

\begin{proof}
By induction on $\zeta$.

The result is trivial for $\zeta=0$.

Suppose that $ \zeta = \delta +1$.
If $\delta \not \in G$,
then the recursive definition \ref{stages}
gives
$ \beta = S^G _ \zeta  = S^G_ \delta + \alpha _ \delta $.    
By the inductive hypothesis, 
 $ S^G_ \delta $ is a mixed sum of  
$(\alpha_ \gamma ) _{ \gamma < \delta } $,
and can be realized by $(A_ \gamma ) _{ \gamma < \delta } $
in such a way that
(1) from \ref{mix} and 
(2$\,'$) from the present lemma are satisfied.
But then, letting $A_ \delta $ be a copy of $\alpha_ \delta $,
and adding  $A_ \zeta $  ``at the top'', we get that
$(A_ \gamma ) _{ \gamma < \zeta  } $ realizes $ \beta = S^G _ \zeta$
and trivially satisfies (1) and  (2$\,'$).

Next suppose that  $ \zeta = \delta +1$ and
$\delta  \in G$.
Then Definition \ref{stages}
gives
$ \beta = S^G _ \zeta  = S^G_ \delta \+ \alpha _ \delta $.    
Again by the inductive hypothesis, 
 $ S^G_ \delta $
 can be realized by some sequence $(A_ \gamma ) _{ \gamma < \delta } $
which satisfies 
(1)  and 
(2$\,'$).
In particular,
$ \bigcup _{ \gamma < \delta } A_ \gamma$
is  the ordinal $S^G_ \delta$.
By the remark after Definition \ref{natsumbl},
$ \beta =  S^G_ \delta \+ \alpha _ \delta $
can be realized as a pure mixed sum
by $A$, $A_ \delta $,
where $A$ has order type $S^G_ \delta$ and
$A_ \delta $ has order type $\alpha_ \delta $. Letting
$\varphi:S^G_ \delta  \to A$.
be the order preserving bijection and  
setting
$A' _ \gamma =  \varphi (A_ \gamma ) $,
for $\gamma < \delta $,
and   $A' _ \delta  =  A_ \delta  $,
then clearly
$(A'_ \gamma ) _{ \gamma < \zeta  } $ is a realization of  $ \beta = S^G _ \zeta$.
Now condition (1) in 
\ref{mix} is satisfied, since, for each $\varepsilon < \delta $,
the construction adds at most one element to 
$\Gamma _ \varepsilon $; indeed, the only element 
which perhaps should be added is $\delta$.
On the other hand, 
$\Gamma _ \delta $   
is empty, thus 
$(A'_ \gamma ) _{ \gamma < \zeta  } $
satisfies (1). 
As far as 
(2$\,'$) is concerned, 
we have by the inductive hypothesis that 
$(A_ \gamma ) _{ \gamma < \delta } $
gives a realization of 
$ S^G_ \delta $
such that  (2$\,'$) is satisfied, thus 
there is a finite $F \subseteq \delta  $
such that  
if $\varepsilon \in \delta   \setminus F$, 
then  $A_ \varepsilon$ is
convex and  contained in some block of $ S^G_ \delta $.
By construction, blocks of $ S^G_ \delta $ are sent by 
$\varphi$  
to blocks of $A$; moreover,
since 
$ S^G _ \zeta  = S^G_ \delta \+ \alpha _ \delta $
is realized as a pure mixed sum
by $A$, $A_ \delta $,
then every block of 
$A$ is contained in some block of
$ S^G _ \zeta$, by Lemma \ref{blockrep}.
Thus
if $\varepsilon \in \delta   \setminus F$, 
then  $A'_ \varepsilon$ is
 contained in some block of $ S^G_ \zeta $.
Moreover, 
if $\varepsilon \in \delta   \setminus F$, then
$A'_ \varepsilon$ is
convex as a subset of $A$; but then
$A'_ \varepsilon$ is also
convex in $ S^G_ \zeta $,
since 
$A'_ \varepsilon$  is convex in $A$,
$A'_ \varepsilon$  is contained in some block of $A$, 
and all the  blocks of $A$ are convex in $ S^G_ \zeta $,
the realization of $ S^G_ \zeta $ being pure.
Hence  (2$\,'$) holds at step $\zeta$ by taking 
$F' = F \cup \{ \delta \} $. 

Suppose now that $\zeta$ is a limit ordinal.
If the sequence $( \alpha _ \gamma ) _{ \gamma < \zeta  } $ is constantly
zero from some point on, 
then the result is immediate from the inductive hypothesis.
Otherwise,
 let $\bar{ \gamma }$ be given by Theorem \ref{infsum}.
By the inductive hypothesis, 
 $ S^G_ {\bar{ \gamma }}$
 can be realized by some sequence $(A_ \gamma ) _{ \gamma < \bar{ \gamma } } $
which satisfies 
(1)  and 
(2$\,'$). 
By equation  \eqref{eqsums} 
in Theorem \ref{infsum},
 $\sum _{ \bar{ \gamma } \leq \gamma < \zeta } \alpha _{ \gamma}
= \omega ^ \xi$.
Then 
the order-theoretical characterization of $\sum $ shows 
 that $\omega ^ \xi$ can be represented as a mixed
sum of 
$(\alpha _{ \gamma})  _{ \bar{ \gamma } \leq \gamma < \zeta } $ 
in such a way that 
all the pieces realizing the mixed sum are convex in the realization
and, moreover, they are disposed in the same order
as the corresponding 
$\alpha _{ \gamma}$'s. 
If we join the two representations by putting
all the elements representing 
$\omega ^ \xi$ above the elements representing
$ S^G_ {\bar{ \gamma }}$,
we get the ordinal 
$ S^G_ {\bar{ \gamma }} + \omega ^ \xi$,
which is equal to 
$ \beta =  S^G_ \zeta $, by Corollary \ref{computeg}.
This new representation clearly satisfies  
(1); indeed, 
$\Gamma _ \varepsilon = \emptyset $,
if $\varepsilon \geq \bar{ \gamma } $,
and
$\Gamma _ \varepsilon  $
remains the same of the sequence $(A_ \gamma ) _{ \gamma < \bar{ \gamma } } $,
if $\varepsilon < \bar{ \gamma } $.
Also
(2$\,'$) is satisfied,
since the new elements of the representation
(those with $\varepsilon \geq \bar{ \gamma } $)
are all contained in the single block
corresponding to 
$\omega ^ \xi$, hence $F$ does not become larger.
Notice that $\omega ^ \xi$ could absorb
some other blocks 
(of $ S^G_ {\bar{ \gamma }}$)
below it, 
but, even in case this happens, 
$\omega ^ \xi$ absorbs the whole of 
such blocks, hence each $A_ \gamma $,
for
$   \gamma < \bar{ \gamma } $,
is contained in a single block anyway,
in the representation of $\beta$. 

 We have finished the proof of 
Lemma \ref{mixlem}, hence of Theorem \ref{mix},
as well.  
\end{proof}

\begin{proof}[Proof of \ref{mixx}] 
Since, as we have noticed before,  
 $ \sumh _{ \gamma < \zeta} \alpha _ \gamma  $
is the particular case of 
$\sumg _{ \gamma < \zeta} \alpha _ \gamma  $
when $G= \zeta $,
we get from Theorem \ref{mix} that  
 $ \sumh _{ \gamma < \zeta} \alpha _ \gamma  $
is a mixed sum of 
$( \alpha _ \gamma ) _{ \gamma < \zeta  } $
and can be realized in such a way that 
(1) and (2) from  \ref{mix} are satisfied, hence (3), too,
is satisfied.
It remains to show that  
 $ \sumh _{ \gamma < \zeta} \alpha _ \gamma  $
is the largest one among those mixed sums that satisfy (3).
Again, this is proved by induction on 
$\zeta$. 

The result is trivial if $\zeta=0$. 

Let $\zeta= \delta +1$ be a successor ordinal and 
 let $ \beta' $ be any mixed sum of 
$( \alpha _ \gamma ) _{ \gamma < \zeta  } $
realized by 
$( A _ \gamma ) _{ \gamma < \zeta  } $ in such a way that 
(3) is satisfied.
Set $A= \bigcup _{ \gamma < \delta } A_ \gamma  $
 and suppose that $A$, as a subset of $\beta'$,  has order type $\beta''$.  
Then, through a suitable bijection, 
$( A _ \gamma ) _{ \gamma < \delta  } $
gives a realization of $\beta''$, and this realization
trivially  satisfies (3), since the original realization
$( A _ \gamma ) _{ \gamma < \zeta  } $  satisfies (3).
By the inductive hypothesis, $\beta'' \leq S_ \delta $,
hence, since the pair $A$, $A_ \delta $ gives a representation
of  $\beta'$  as a mixed sum
of   $\beta''$ and $\alpha_ \delta $,
we get, by Carruth Theorem and Definition \ref{iterated}, that 
$ \beta ' \leq \beta'' \+\alpha_ \delta \leq  S_ \delta  \+ \alpha _ \delta = S_ \zeta $. 
The successor step has thus been proved.

Let $\zeta$ be a limit ordinal.
If the sequence $( \alpha _ \gamma ) _{ \gamma < \zeta  } $ is constantly
zero from some point on, 
then the result follows trivially from the inductive hypothesis,
hence we can suppose that 
  $( \alpha _ \gamma ) _{ \gamma < \zeta  } $ is not eventually zero.
 Let again $ \beta' $ be any mixed sum of 
$( \alpha _ \gamma ) _{ \gamma < \zeta  } $
realized by 
$( A _ \gamma ) _{ \gamma < \zeta  } $ in such a way that 
(3) is satisfied.
Since the sequence is not eventually zero
and $\zeta$ is a limit ordinal, then, by (3),
$\beta'$, too,  is a limit ordinal.
Hence it is enough to show that,
for every $\xi < \beta '$, we have $\xi \leq S_ \zeta $.  

So let $\xi < \beta '$, say,
$\xi \in A_ \varepsilon  $. 
By (3), the set
$ \{ \gamma < \zeta  \mid \gamma > \varepsilon 
\text{ and } b< \xi, \text{ for some } b \in A_ \gamma   \}$  is finite;
enumerate it as $\gamma_0, \dots, \gamma $
and let $C_{ i} =  A _{ \gamma _i} \cap [0, \xi)$,
for $i=0, \dots, n$. 
Setting $C  = \bigcup _{ \gamma < \varepsilon  } A_ \gamma  $,
we have that $\xi$ is a mixed sum of $C$,  $C_0$, \dots $C_n$.  
For each $i$, if $\beta_i$ is the order type of $C_i$, 
then $\beta_i \leq \alpha  _{ \gamma _i}$,
since the latter is 
the order type
of  $ A _{ \gamma _i} $ and $C_i \subseteq  A _{ \gamma _i} $.
Moreover, by the inductive hypothesis, 
if $\beta''$ is
the order type of $C$, then  
$ \beta '' \leq S_ \varepsilon  $. 
Since $\xi$ is a mixed sum of $C$,  $C_0$, \dots $C_n$,
then, again by Carruth Theorem, 
$\xi
 \leq 
\beta '' \+ \beta_0 \+ \dots \+ \beta _n 
\leq 
S_ \varepsilon  \+  \alpha  _{ \gamma _0} \+ \dots \+  \alpha  _{ \gamma _n}
\leq 
S_ \zeta $,
 where, in order to get the last inequality, notice that, by construction,
$\gamma _0, \dots, \gamma _n > \varepsilon $,
hence 
the inequality follows from 
Proposition \ref{factsww}(5). 
 \end{proof}   
 
Notice that the proof also shows that 
$ \sumh _{ \gamma < \zeta} \alpha _ \gamma  $
can be realized in such a way that
(1) from \ref{mix} and (2$'$) from \ref{mixlem}
are satisfied.

\section{Invariant infinite natural sums} \labbel{furth}

\subsection*{The problem of invariance under permutations} 
All the previous notions and  results
are dependent on the order in which the 
$\alpha_ \gamma  $'s appear in the sequence
$(\alpha_ \gamma ) _{ \gamma < \zeta } $.
In particular, for $\zeta > \omega $,
the value of  $ \sumh _{ \gamma < \zeta} \alpha _ \gamma  $
is generally \emph{not} invariant under permutations
(invariance holds, however, for a very special class of permutations,
as we shall mention in Proposition \ref{factswww} below).
To see that a transfinite sum is not necessarily
invariant, just take $\zeta= \omega +1$, $\alpha_0= 0$
and   $\alpha_ \gamma = 1$, for $0 < \gamma < \omega +1$,
thus   $ \sumh _{ \gamma < \zeta} \alpha _ \gamma = \omega +1 $.
On the other hand, if we permute $\alpha_0$ and 
$\alpha_ \omega $, that is, we  reindex the 
$\alpha_ \gamma  $'s as $(\beta_ \gamma ) _{ \gamma < \omega +1} $  by letting
$ \beta _ \gamma = 1$, for $0 \leq \gamma < \omega $
and $\beta _{ \omega } =0 $, then
we get  $ \sumh _{ \gamma < \zeta} \beta _ \gamma = \omega \not= \omega +1 $.

Noninvariance of $\sumh$  under permutations 
strongly contrasts with \cite{w},
that is, with the case of $ \omega$-indexed sequences.
In fact, essentially all the results in \cite{w} are independent
from the chosen ordering of the $\alpha_i$'s and
some results there do not even mention 
the ordering.  For example, in \cite[Theorem 4.7]{w} we proved
that the natural sum of an $ \omega$-indexed sequence of ordinals
is the maximum 
of all the left-finite mixed sums
of the ordinals in the sequence.   
Recall the definition of a mixed sum from 
\ref{infmix}. The definition of left-finiteness
is recalled in the next definition. 

\begin{definition} \labbel{lf}    
A mixed sum $\beta$ of $( \alpha _i) _{i \in I}$  is \emph{left-finite} 
if it
can be realized by $(A_i) _{i \in I}$ 
in such a way that, for every $a \in \beta$,
 the set of all the elements smaller than $a$ 
is contained in the union of a finite number of $A_ i$'s.
\end{definition}

Notice that, in the specific case of  an $ \omega$-indexed sequence, 
left-finiteness is equivalent to condition (3) in Corollary \ref{mixx}; moreover,
 since we are assuming that $\zeta= \omega $,
 the condition is independent 
from the ordering of the sequence. 
Hence Theorem \ref{wt} can be obtained as a special
case of Corollary \ref{mixx}.
Notice also that it  is by no means trivial that
the set of all the left-finite mixed sums 
of some given $ \omega$-indexed sequence
has a maximum, not simply a supremum.

A na\"\i ve  approach 
in search of
 a generalization of the above mentioned 
Theorem 4.7 from \cite{w}
for, say, a sequence indexed by 
a set of cardinality $ \omega_1$, would be
to 
restrict oneself to 
\emph{left-countable} mixed sums, that is, asking that,
for every element $a \in \beta $, the set of all the elements smaller than $a$ 
is contained in the union of countably many $A_ \gamma $'s.
Here and below, by \emph{countable} we mean either finite
or denumerably infinite. 
However, fixed an $ \omega_1$-sequence of ordinals,
the set of the left-countable mixed sums of the sequence
might not have a maximum.
Just take $\alpha_ \gamma = \omega_1 $, for $\gamma < \omega $,
and   $\alpha_ \gamma = 1 $, for $\omega \leq \gamma < \omega _1$.
Every ordinal of the form $ \omega_1 \varepsilon $,
for $0 \not=\varepsilon < \omega _1$, is a  left-countable 
 mixed sum of 
$(\alpha_ \gamma ) _{ \gamma <  \omega _1 } $
but this is not the case for the supremum of the above values, i. e., 
$ \omega_1 ^2$. 
If $\varepsilon> \omega $, then
in the above example we can even realize the  left-countable
 mixed sum in such a way that all pieces are
convex.

The above example suggests that 
it will be difficult, or perhaps impossible, to find 
some natural infinitary generalization of the Hessenberg sum
for sequences indexed by a set which is not supposed to be (well-)ordered; or,
put in another way, that the countable case of the infinitary natural
sum is very special and, usually, results do not generalize 
to uncountable cardinals. In this respect, see also the remarks on
\cite[p. 370]{VW} and our review \cite{zbl} of \cite{VW}.
  In particular, it seems difficult to find some 
infinitary operation (on uncountably many arguments) which has some good
 purely order-theoretical characterization
and which does not rely on the ordering of the sequence.
See, however, \cite{ans} for a possible alternative approach 
to the problem. 
Notice that, on the other hand, when the sequence is well-ordered, 
Corollary  \ref{mixx} 
provides such an order-theoretical characterization
for the iterated natural sum of Definition \ref{iterated}.

\subsection*{Some invariant sums} 
Since,  in general, by the above example, it is probably not always 
possible to find
some kind of ``maximal sum'', we can at least 
define some minimal ones (which, by the very definition, will turn out to be 
automatically invariant under permutations).

\begin{definitions} \labbel{nats}
Suppose that $\zeta$ is an ordinal and
$( \alpha _ \gamma ) _{ \gamma < \zeta } $
is a sequence of ordinals.
Define
\begin{equation*}    
\sideset{}{^o}\nsum _{ \gamma < \zeta  } \alpha _ \gamma  = 
\inf_ \pi \sumh _{ \gamma < \zeta  }  \alpha _{ \pi( \gamma )} 
 \end{equation*} 
where $\pi$ varies among all the permutations of $\zeta$.
By a \emph{permutation of $\zeta$} we mean a
 bijection from $\zeta$ to $\zeta$.
Notice that in the above definition we are keeping $\zeta$ fixed.
Allowing $\zeta$ to change, we generally obtain different results;
for example, if $\alpha_ \gamma = 1$, for every $\gamma$, 
then $\sumh _{ \gamma < \zeta  } \alpha _ \gamma  =
\sum _{ \gamma < \zeta  } \alpha _ \gamma =
 \zeta $,
for every $\zeta$.  In the next
 definition, on the contrary, we let the ordinal vary.

Suppose that $I$ is \emph{any} set and
$( \alpha_i) _{i \in I} $ is a sequence of ordinals.
Define
 \begin{equation*}   
\nsum _{ i \in I } \alpha _i = 
\inf_{ \zeta , f} \sumh _{ \gamma < \zeta  }  \alpha _{ f( \gamma )} 
 \end{equation*} 
where $\zeta$ varies among all the ordinals
having cardinality $|I|$ and 
$f$ varies among all the bijections from $\zeta$ to $I$.
In this situation, we shall call $f$ a \emph{rearrangement}
of the sequence.   

Furthermore, 
let $\lambda=|I|$ and define
 \begin{equation*}  
\sideset{}{^\bullet}\nsum _{ i \in I } \alpha _i = 
\inf_{ f} \sumh _{ \gamma < \lambda  }  \alpha _{ f( \gamma )} 
 \end{equation*} 
where $f$ varies among all the bijections from $ \lambda $ to $I$.
 \end{definitions}   

The difference between 
$\nsum$ and $\nsum^\bullet$ is that in 
$\nsum $ we consider rearrangements into a sequence
of arbitrary length, while    
in 
$\nsum^\bullet$ we consider only rearrangements into a sequence
of length $|I|$. On the other hand,
in $\nsum ^o$ the length of the sequences is assumed to be fixed. 

For every $I$,
trivially, 
$\nsum _{ i \in I } \alpha _i \leq \nsum^\bullet_{ i \in I } \alpha _i$ 
and
if $|I| = | \zeta |$
and 
$( \alpha_i) _{i \in I} $  is a rearrangement of 
$( \alpha _ \gamma ) _{ \gamma < \zeta } $,
 then 
$  \nsum _{ i \in I } \alpha _i \leq \nsum ^o _{ \gamma < \zeta  } \alpha _ \gamma $.
Of course, we could have written the above inequality simply as
 $\nsum  _{ \gamma < \zeta  } \alpha _ \gamma  \leq \nsum ^o _{ \gamma < \zeta  } \alpha _ \gamma $, with no need of introducing rearrangements,
since an ordinal is, in particular, a set. However, it seems clearer
to use a letter such as $I$ for a set on which no particular order is defined,
and we shall usually obey this convention.
Though, as just mentioned, $\nsum$ is always $\leq$ than both  $ \nsum^\bullet$
and $ \nsum ^o$,    
on the other hand, in general,
there is no provable inequality
between
$ \nsum^\bullet$
and $ \nsum ^o$. Indeed, if
$\alpha_ \gamma = 1$, for every $\gamma < \omega +1$,   
then
$ \nsum^\bullet _{ \gamma < \omega +1  } \alpha _ \gamma = \omega <
\omega +1=
 \nsum ^o _{ \gamma < \omega +1  } \alpha _ \gamma $.
In the other direction, we shall show in the last 
sentence  in Example \ref{count}
that $ \nsum ^o$ can be strictly smaller than  $ \nsum^\bullet$.  

In any case, all the above operations coincide
for $ \omega$-indexed sequences. 
If $\zeta= |I|= \omega  $ and
$( \alpha_i) _{i \in I} $  is a rearrangement of 
$( \alpha _ \gamma ) _{ \gamma < \omega  } $,
then
\begin{equation}\labbel{idid}
 \sideset{}{^o}\nsum  _{ \gamma < \omega } \alpha _ \gamma 
=
 \sumh  _{ \gamma < \omega } \alpha _ \gamma 
=
\sideset{}{^\bullet}\nsum _{ i \in I } \alpha _i 
=
 \nsum _{ i \in I } \alpha _i
  \end{equation}    
 This shows that the 
notation in Definitions \ref{nats}
is consistent both with \cite{VW} and with 
\cite{w}.  
The first two identities in \eqref{idid} follow from
\cite[p. 362]{VW} or
\cite[Proposition 2.4(5)]{w}, to the effect that the $ \omega$-indexed
 natural sum is invariant under permutations,
a fact which shall be generalized in Proposition \ref{factswww} below. 
Notice that  invariance of 
$ \sumh  _{ \gamma < \omega } \alpha _ \gamma $ 
under permutations is also an immediate consequence 
of the order-theoretical characterization given 
in \cite[Theorem 4.7]{w} and mentioned in the
preceding subsection.
The last identity in \eqref{idid} shall be proved
in Proposition \ref{id} below. 
Before giving the proof, we provide a counterexample 
that shows the reason why the proof is not entirely trivial.

\begin{example} \labbel{count}   
It is somewhat surprising that
it is possible to have the strict inequality
$\nsum _{ i \in I }\alpha _i < \nsum _{ i \in I }^\bullet  \alpha _i$
in case $I$ is uncountable.
This is somewhat counterintuitive, since, allowing a longer sequence,
 we might get a smaller outcome.
Let $|I|= \omega _1$,
$\alpha_i \in \{ 1, \omega _1 \} $, for $i \in I$,
$|\{ i \in I \mid \alpha _i = 1 \}| = \omega _1$ and
$|\{ i \in I \mid \alpha _i =  \omega _1 \}| = \omega $.
For every bijection $f: \lambda \to I$, we have 
$\sumh _{ \gamma < \lambda  }  \alpha _{ f( \gamma )}
\geq  \omega_1( \omega +1) $,
in fact, we already have
$\sum _{ \gamma < \lambda  }  \alpha _{ f( \gamma )}
\geq  \omega_1( \omega +1) $.
Recall that sums and products are 
always intended in the ordinal
sense.
Thus $\nsum _{ i \in I }^\bullet \alpha _i \geq \omega_1( \omega +1) $,
in fact, 
$\nsum _{ i \in I }^\bullet \alpha _i = \omega_1( \omega +1) $,
the example giving the reverse inequality being easy. 

However, 
$\nsum _{ i \in I } \alpha _i  = \omega_1 \omega $. 
Indeed, let $\zeta= \omega _1 + \omega $,
$\beta_ \gamma = 1$, for $\gamma < \omega _1$,
and    
$\beta_ \gamma = \omega _1$, for $ \omega _1 \leq  \gamma < \omega _1 + \omega $.
Then $\sumh _{ \gamma < \omega _1 + \omega} \beta _ \gamma = 
\sum _{ \gamma < \omega _1 + \omega} \beta _ \gamma
= \omega_1 \omega$.
Since the $\beta_ \gamma $'s are a rearrangement of 
the $\alpha_i$'s, we get
$\nsum _{ i \in I } \alpha _i \leq  \omega_1 \omega $.
The reverse inequality is obvious.  

Notice that we have also showed
that
$ \nsum ^o_{ \gamma < \omega _1 + \omega} \beta _ \gamma = 
 \omega_1 \omega <  \omega_1( \omega +1)  = \nsum _{ i \in I }^\bullet \alpha _i $.
 \end{example}

Similar counterexamples  are well-known,
when  corresponding definitions are considered
relative to the usual ordinal sum
$\sum$ in place of $\sumh$.
See Rado \cite{Ra}, in particular, p. 219 therein,
where a counterexample is given even in the countable case.
It is quite interesting that, on the other hand, 
the natural sum is immune to such counterexamples, 
as far as countable sums are taken into account.

\begin{proposition} \labbel{id}
If $|I| \leq \omega $,
then $\nsum _{ i \in I }^\bullet \alpha _i 
=
 \nsum _{ i \in I } \alpha _i
$. 
 \end{proposition}

  \begin{proof}
If $I$ is finite, the definitions are clearly the same, so let  $|I| = \omega $.
As we mentioned above, the inequality 
$\nsum _{ i \in I } \alpha _i \leq \nsum^\bullet_{ i \in I } \alpha _i$
is trivial; moreover,
$\nsum _{ i \in I }^\bullet \alpha _i 
=
 \sumh_{ \gamma < \omega } \alpha _ \gamma $,
for every rearrangement of the $\alpha_i$'s
into a sequence of length $ \omega$.
We have to show that if $\zeta$ is a countably infinite ordinal
and $( \beta _ \delta ) _{ \delta < \zeta }  $ 
is another rearrangement of the  
$\alpha_i$'s, this time
into a sequence of length $ \zeta $,
then
$ \sumh_{ \gamma < \omega } \alpha _ \gamma  
\leq
  \sumh_{ \delta  < \zeta  } \beta  _ \delta  $. 

Let $ \xi$ be the smallest ordinal
such that $\{ i \in I \mid \alpha _i \geq \omega ^\xi\}$
is finite.
Enumerate those $\alpha_i$'s such that 
 $\alpha _i \geq \omega ^\xi$ as
$ \alpha _{i_0}, \dots, \alpha _{i_h}  $
(the sequence might be empty).
If $\xi=0$, then all but a finite number
of the $\alpha_i$'s are zero and the result follows easily
from the finite case (e.~g., use Proposition \ref{factsww}(4)). 
If  $\xi>0$, then equation (6) in 
\cite[Corollary 5.1]{w} gives
$
 \nsum^\bullet_{ i < \omega } \alpha _i
=
 \sumh_{ \gamma < \omega } \alpha _ \gamma  
=
 \alpha _{i_0}^{\restriction  \xi} \+ \dots \+ \alpha _{i_h}^{\restriction  \xi}
\+ \omega ^\xi$
(recall the definition of $\alpha^{\restriction  \xi}$
given right before Proposition \ref{facts}). 

Turning to the 
rearrangement
$( \beta _ \delta ) _{ \delta < \zeta }  $,
let 
 $ \beta  _{ \delta _0}, \dots, \beta  _{ \delta _h}  $,
with
$\delta_0 < \dots < \delta _h$, 
be an enumeration of those $\beta_ \delta $'s such that  
$ \beta _ \delta  \geq \omega ^\xi$
(of course,  $ \beta  _{ \delta _0}, \dots, \beta  _{ \delta _h}  $
is a rearrangement of 
 $ \alpha   _{ i _0}, \dots, \alpha   _{ i_h}  $).
Now the proof splits into two cases.
First, suppose that $\xi$ is a successor ordinal, say,
$\xi= \xi' +1$. Then, by the very definition of $\xi$,
there are infinitely many $\beta_ \delta $'s
such that  
$ \omega ^{\xi} > \beta _ \delta  \geq \omega ^{\xi'}$.
Hence we can choose a subsequence of 
$( \beta _ \delta ) _{ \delta < \zeta }  $ of order type $ \omega$
and consisting of 
elements $ \geq \omega ^{\xi'}$ and $< \omega ^{\xi}$.
Define another sequence
$( \beta' _ \delta ) _{ \delta < \zeta }  $
obtained from $( \beta _ \delta ) _{ \delta < \zeta }  $
by leaving unchanged  the values of the elements
of the above subsequence,
by leaving unchanged  the values of 
 $ \beta  _{ \delta _0}, \dots, \beta  _{ \delta _h}  $, as well,
and turning to $0$ all the other values.
By Proposition \ref{factsww}(2),
 $  \sumh_{ \delta  < \zeta  } \beta  _ \delta  \geq
  \sumh_{ \delta  < \zeta  } \beta ' _ \delta $. 
  By Proposition \ref{factsww}(4),
 $   \sumh_{ \delta  < \zeta  } \beta ' _ \delta 
=
  \sumh_{ \varepsilon  < \omega + k  } \beta '' _ \varepsilon $,
where $(\beta'' _ \varepsilon ) _{\varepsilon  < \omega + k  } $
is the subsequence of the nonzero $\beta' _ \delta $'s, thus  
 $k$ is finite. 
Then, applying again  
 \cite[Corollary 5.1]{w}, 
we get, for some $j \leq h$
(in fact, $j$ is such that $j + k = h$), 
$ \sumh_{ \delta  < \zeta  } \beta  _ \delta  \geq
\sumh_{ \varepsilon  < \omega + k  } \beta '' _ \varepsilon =
\beta  _{ \delta _0}^{\restriction  \xi} \+ \dots \+ \beta  _{ \delta _j}^{\restriction  \xi}
\+ \omega ^{\xi} \+  \beta  _{ \delta _{j+1} } \+ \dots
\+ \beta  _{ \delta _{h} } \geq
\beta  _{ \delta _0}^{\restriction  \xi} \+ \dots \+ \beta  _{ \delta _j}^{\restriction  \xi}
\+ \omega ^{\xi} \+  \beta  _{ \delta _{j+1} }^{\restriction  \xi} \+ \dots
\+ \beta  _{ \delta _{h} }^{\restriction  \xi}
=
 \alpha _{i_0}^{\restriction  \xi} \+ \dots \+ \alpha _{i_h}^{\restriction  \xi}
\+ \omega ^\xi =
 \sumh_{ \gamma < \omega } \alpha _ \gamma$,
what we had to show. 

The case when $\xi$ 
is limit is similar. This time, choose some subsequence
of $( \beta _ \delta ) _{ \delta < \zeta }  $
of type $ \omega$ 
in such a way that, for every $\xi' < \xi$,
there is some element of the subsequence which is
$ \geq \omega ^{ \xi'} $
and $ < \omega ^{ \xi} $
(notice that if $\xi$ is limit,
then necessarily $\xi$ has cofinality $ \omega$).

All the rest goes  the same way. 
 \end{proof}

Notice that the counterexample in \ref{count} 
shows also that $\nsum^\bullet$ is not invariant under extending a sequence
by adding further $0$'s, while $\nsum$ is indeed invariant in this sense.
More formally, if $( \alpha_i) _{i \in I} $ is a sequence of ordinals,
$J \supseteq I$ and we set $ \alpha _i =0$, for $i \in J \setminus I$,
then   $\nsum _{ i \in I } \alpha _i = \nsum _{ i \in J} \alpha _i$, 
as a consequence of 
 Proposition \ref{factsww}(4).  The analogous identity fails
for $\nsum^\bullet$. Just consider the sequence $( \alpha_i) _{i \in I} $ 
from \ref{count} 
 and let $|J| \geq \omega_2$. Then the arguments in 
\ref{count}  show that
$\nsum^\bullet _{ i \in I } \alpha _i = 
\omega_1( \omega +1 ) \not= \omega _1 \omega =
\nsum^\bullet _{ i \in J} \alpha _i$.

The above remark suggests that $\nsum$ is perhaps a more natural operation
than $\nsum^\bullet$.

\subsection*{Invariance in special cases} 
As promised, we now show 
that $\sumh$ is invariant under a special class of permutations.
First, a  definition is in order.
If $\zeta$ is an ordinal, let us call a subset $A$  of
$\zeta$ a \emph{component} of $\zeta$
if $A$ has either the form
$[ \alpha, \alpha + \omega )$ or
$[ \alpha, \zeta  )$,
where in both cases either $\alpha=0$ or $\alpha$ is a limit ordinal. 
 Thus the components partition
$\zeta$. Of course, there is just one component 
of  the kind  $[ \alpha, \zeta  )$, all the others have length $ \omega$.
With the above definition, we can show that 
$\sumh$ is invariant under the (somewhat special kind of)  permutations
which act on each component.
Moreover, a form of the general associative-commutative law
holds in some special cases, to the effect that, besides performing
the above kinds of permutations, we can associate sets of \emph{finitely many} 
 elements \emph{inside the same component}.

\begin{proposition} \labbel{factswww}
  \begin{enumerate}[(1)]    
\item  
If  $\pi$ is a permutation of $ \zeta $
such that $\pi(C) = C$,
for every component $C$  of $\zeta$, 
 then
$\sumh _{ \gamma  < \zeta  } \alpha _ \gamma  
= 
\sumh _{ \gamma  < \zeta  } \alpha _{ \pi ( \gamma )} $.
\item
More generally, suppose that 
$(F_h) _{ h < \zeta '} $ is a partition of $\zeta$ such that 
  \begin{enumerate}[(a)]
    \item 
each $F_h$ is finite, say, $F_h = \{ \delta _1, \dots, \delta  _{r(h)} \} $,
\item
each $F_h$ is a subset of some component $C_h$ of $\zeta$,
\item
the $F_h$'s are ordered in such a way that 
if $C_h$ occurs before $C _{h'} $ in $\zeta$, then $h < h '$;
that is, the ordering of the $F_h$'s respects the ordering of the components,
but, inside a component, the ordering of the $F_h$'s can be arbitrary. 
  \end{enumerate}
   Then
\begin{equation*}\labbel{ca}     
\sumh _{ \gamma  < \zeta  } \alpha _ \gamma  
= 
\sumh _{ h  < \zeta'  }   
\,\pl _{ \delta \in F_h}  \alpha _ \delta =
\sumh _{h < \zeta '  } ( \alpha _{ \delta _1} \+ \alpha _{ \delta _2} \+ \dots 
\+ \alpha _{ \delta _{r(h)}}  )
 \end{equation*}
  \end{enumerate}
 \end{proposition} 

\begin{proof}
The proposition has an elementary
proof similar to \cite[Proposition 2.4(5)(6)]{w}.

The proposition can be  given also an order-theoretical proof, 
using Corollary  \ref{mixx}.
As far as (1) here is concerned, just notice that
 finiteness of the sets in \ref{mixx}(3) 
 is preserved under the 
permutations at hand.

As for (2), let $\beta_h = \pl _{ \delta \in F_h}  \alpha _ \delta$,
for 
$h < \zeta '  $.
Then, applying Corollary \ref{mixx}
to   
$\sumh _{h < \zeta '  } \beta_h$,
we get a mixed sum of the $\beta_h$'s
which satisfies condition   \ref{mixx}(3).
Expanding the $\beta_h$'s using Carruth's theorem
(this is possible by (a)),
we get a mixed sum of the $\alpha_ \gamma $'s,
and this sum satisfies \ref{mixx}(3), by the assumptions
(b) and (c). Thus, by Corollary \ref{mixx},
$\sumh _{ \gamma  < \zeta  } \alpha _ \gamma  
\geq
\sumh _{ h < \zeta'  } \beta _h  $.

Conversely, apply Corollary \ref{mixx}
to  $\sumh _{ \gamma  < \zeta  } \alpha _ \gamma  $.
In the mixed sum given by \ref{mixx},
for each $\delta \in F_h$, 
join together 
$ A _{ \delta _1},  A _{ \delta _2}, \dots 
A _{ \delta _{r(h)}}$, and call
$B_ \delta $ this union.
Again by Carruth theorem and (a), 
the order type $\beta'_ \delta $ of  
$B_ \delta $ is $ \leq \beta_h = \pl _{ \delta \in F_h}  \alpha _ \delta$.
Since the $B_ \delta $'s realize a mixed sum
of the $\beta'_ \delta $'s, and this realization satisfies 
\ref{mixx}(3), by 
(b) and (c),
we get
$
\sumh _{ \gamma  < \zeta  } \alpha _ \gamma
\leq
\sumh _{ h < \zeta'  } \beta' _h
\leq
\sumh _{ h < \zeta'  } \beta _h $,
by Corollary \ref{mixx} and  
Proposition \ref{factsww}(2).
\end{proof}

\begin{problems} \labbel{prob} 
(a)
The iterated natural sum can be extended to the surreal numbers,
in a way we are going to explain soon.
See Conway \cite{C} for details about surreal numbers and, e.~g., Siegel \cite{Sieg}
for an updated list of references.
A surreal number $s$ can be thought of as an ordinal-indexed string consisting of
$+$ and  $-$'s; this is called the \emph{sign expansion} of $s$.  
The ordinals can be considered as a
 substructure of the surreals; in this sense, an ordinal is a surreal
 having only $+$'s in its sign expansion.
The surreal sum, when restricted to the ordinals,
does correspond to the ordinal natural sum.   
One can also define the limit of a transfinite sequence of surreals;
see Mez\H{o} \cite{M} and \cite{arx}. Roughly, the limit of
an ordinal-indexed sequence of surreals is the longest string $s$ 
such that every initial segment of $s$ 
is eventually coincident with the corresponding (possibly improper)
 initial segments
of the members of the sequence
(we are allowing the length of $s$ 
to be a successor ordinal, in which case
$s$ is required to be eventually an initial segment
of the members of the sequence).
 Notice that the limit $s$ might be much shorter
than the superior limit of the lengths of the members of the sequence,
actually, 
$s$ can be the empty sequence!
Then Definition \ref{iterated} extends to the surreals.
See \cite{arx}  for full details. 

Which results from the present paper and  from \cite{w}
generalize to this surreal iterated sum? 

(b) Conversely, an ordinal sum can be defined within the surreals. 
In the sense of string expansions, it corresponds to
string concatenation; see Conway \cite[Chapter 15, p. 193]{C}.
It can be obviously iterated through the transfinite.
Which results about transfinite ordinal sums (of ordinals) do generalize
to the surreals?

Most of the  problems which follow can be extended to the surreals, too.

(c) Though, in general, the iterated natural sum $\sumh$ from \ref{iterated} 
is not invariant under permutations, one might ask for which
sequences $(\alpha_ \gamma ) _{ \gamma  < \zeta } $  of ordinals 
 the sum $\sumh_{ \gamma  < \zeta  }  \alpha _ \gamma $ 
turns out to be indeed  invariant under permutations.
The corresponding problem for
the usual transfinite ordinal sum has been studied,
see Hickman \cite{H} and further references there.
Of course, for the iterated natural sum this kind of
``generalized commutativity''
  is a much more frequent phenomenon, since it holds 
 for all finite and $ \omega$-indexed sequences.

(d) In particular, under which conditions (on an ordinal-indexed 
sequence of ordinals)
do some of the operations 
$\sumh$,
$\nsum ^o$, $\nsum$,  
$\nsum^\bullet$ and
 $\sum$ give the same outcome?
One can also take into account the operations
$\sum ^o$, $\sum^*$  and
$\sum ^\bullet$, which are defined as in Definitions \ref{nats},
by replacing everywhere $\sumh$ by $\sum$
($\sum^*$  corresponds to $\nsum$). 
The operation $\sum ^*$ has been studied quite thoroughly,
see, e.~g., Rado \cite{Ra}, Anderson \cite{An}.
Notice that sometimes in the literature the word
\emph{permutation} is used to mean what 
we call here a \emph{rearrangement}. 
As we mentioned, Rado \cite[p. 219]{Ra} shows 
that $\sum^\bullet$  and $\sum^*$ might give different outcomes.
On the other hand, the operation 
$\sum ^o$, when the index set is not a cardinal, 
  seems to have received less  attention.

(e) Of course, there is a more general formulation of Problem
 (c) above, 
asking how many values $\sumh$ assumes
when we permute (or, more generally, rearrange) the elements of some given sequence. 
In the case of $\sum$ the corresponding problem
has been studied; see, e.~g., Sierpi\'nski \cite{Sie}, 
Ginsburg \cite{G}, Hickman \cite{H1}, 
Komj{\'a}th \cite{K} 
and further references in these papers.

(f) The fact that we do not always have 
  ``maximal sums'' for $\sumh$, i. e., that some suprema are not necessarily 
attained,
 leaves out the possibility 
of the existence of maximal sums for special kinds of sequences.
For $\sum$ this has been studied; see 
 Dushnik  \cite{D} and  Anderson
\cite{A2}.

(g) Study transfinite natural products defined in the same vein as of Definitions
\ref{iterated}, \ref{stages} and \ref{nats}. Ideas from Altman \cite{A} 
might be relevant to the problem. 
Here order-theoretical characterizations will be probably much harder to come by.
All the  problems above can be asked for infinite natural products, too.

For some properties of ordinary (not ``natural'') transfinite products
see \cite[XIV, 17]{Sier} and \cite[III, \& 10]{Bac}.

(h) By applying the characterization
of $ \sumh _{ \gamma < \zeta} \alpha _ \gamma  $ given in 
Corollary  \ref{mixx}, one can surely rephrase 
Definitions \ref{nats} in order to provide 
order-theoretical characterizations of $\nsum ^o$, $\nsum$  
and $\nsum^\bullet$. Such characterizations appear muddled, complicated
and far from being useful.
Are there simpler and more useful
order-theoretical characterizations of
$\nsum ^o$, $\nsum$,  
$\nsum^\bullet$?
\end{problems}

\section{Notions of size for well-founded trees} \labbel{not} 
Wang \cite{W} and  V\"a\"an\"anen and Wang \cite{VW} 
defined  notions of size for an 
$\mathcal L _{ \omega _{1}, \omega }$-formula in negation normal form.
Recall that $\mathcal L _{ \omega _{1}, \omega }$ is the extension
of first-order logic in which countable disjunctions and conjunctions
are allowed.
Since a (possibly infinitary) formula can be viewed  as a 
labeled well-founded tree and V\"a\"an\"anen and Wang's definition
 depends
only on the tree structure,
not on the labels,
they implicitly give  definitions of size for 
countable
 well-founded trees 
(to be pedantic,  for those trees arising from 
$\mathcal L _{ \omega _{1}, \omega }$-formulas
in negation normal form; notice also that,
in the definition of size from \cite{VW}, negating an atomic formula
does not augment size). 
By extending their ideas and using Definitions \ref{nats},
we can provide notions of size which apply to every well-founded tree,
not only to countable ones. 

Here we intend a  tree
in the classical set-theoretical sense
but
we shall describe it in terms of the reversed order.
A \emph{(reversed) well-founded tree} is a 
well-founded partially ordered set $(T, \leq)$ such that, 
for every $t \in  T$, the set of all successors of $t$ 
is finite and linearly ordered.
By, e.~g., \cite[Theorem 2.27]{Je},
every element $t$ of a well-founded partially ordered set
has a well-defined \emph{rank} $\rho(t)$;
the rank of   $t$ is the smallest ordinal which is
strictly larger than all the ranks of the predecessors
of $t$. This justifies  inductive definitions
on ranks.
Ranks go the other direction with respect to \emph{levels};
maximal elements are at level $0$ but 
if the tree has just one maximal element (the \emph{root})
this is the element of largest rank. 
When the order $\leq$ is understood, we shall simply write
$T$ in place of $(T, \leq)$.

\begin{definition} \labbel{size}
If $T$ is a well-founded (reversed) tree, then, for every $t \in T$, 
we define the \emph{size} $\sigma(t)$ of $t$ 
by induction on the rank of $t$ as 
$\sigma(t) = \left(\nsum _{ u \in P(t)}\sigma(u)\right) + 1$,
where $P(t)$ is the set of all the immediate predecessors of $t$.
In particular, minimal elements of $T$ have size $1$.

 The \emph{size} $\sigma(T)$ of $T$ is 
$\sigma(T) = \nsum _{ t \in M}\sigma(t)$,
where $M$ is the set of the maximal elements of $T$.
In particular, if $T$ has a unique root, 
then $\sigma(T)$ is 
the size of the root of $T$.

Notice that when $T$ is finite the above defined size gives the 
cardinality of $T$ (the number of its nodes). 

Similar definitions can be given using
$\nsum ^o$ or   
 $\nsum^\bullet$
in place of $\nsum$.
 \end{definition}   

In particular, since a formula 
of a (possibly infinitary) logic can be seen as a (labeled) well-founded tree,
the above definition furnishes a possible definition of
size for  a formula.

The size of a countable well-founded tree can be
given an order-theoretical characterization,
as we are going to show, after some preliminary definitions.

If $(T, \leq)$ is a partially ordered set and $t \in T$,
we let 
${\downarrow} t = \{ u \in T \mid u \leq t \}$.
We shall frequently consider another  order $\leq'$  on $T$; 
usually $\leq'$  will be an \emph{extension} of $\leq$,
that is, 
$ u \leq t$ 
implies 
$ u \leq' t$ 
for every $u, t \in T$.
 In the above situation, we shall denote
the set $  \{ u \in T \mid u \leq' t\}$
by 
${\downarrow}'t $.

\begin{definition} \labbel{otc}
If $(T, \leq)$ is a partially ordered set,
we say that  $\leq'$ is \emph{a downward-finite
extension of} $\leq$ if $\leq'$ is an extension of $\leq$
and, for every $v\in T$, there are a finite number
$u_0, \dots u_n$  of  elements of $T$
which are $\leq$-incomparable with $v$ and
 such that
${\downarrow}'v \subseteq {\downarrow}v \cup 
{\downarrow}u_0 \cup \dots \cup {\downarrow}u_n$.

If $\leq$ is understood, we shall simply say that
$\leq'$ is downward finite.
 \end{definition}    

It is well-known that every well-founded partial order
can be extended to a well-order. However, even  in the case of 
a well-founded tree, the order-types of extensions might be unbounded;
just consider an infinite antichain $C$. It has well-ordered extensions 
of every order-type having cardinality $|C|$,
and the supremum of these order-types is $|C|^+$, which  is not attained.
  
On the other hand, 
we are going to show that a maximum extension exists
if we restrict ourselves to downward-finite
well-ordered extensions of countable well-founded trees; moreover,
the order-type of this extension is exactly the size of the tree,
as introduced in Definition \ref{size}.
In our opinion, this result shows the naturalness 
(at least in the countable case) both of the
definition of $\nsum$ and of the above notion of 
size for a well-founded tree.

\begin{theorem} \labbel{wft}
If $(T, \leq)$ is a countable well-founded (reversed) tree, then
$\leq$ has a  downward-finite extension which is a well-order
of type $\sigma(T, \leq)$.

Moreover, every downward-finite well-order extending
$\leq$ has order-type less than or equal to $\sigma(T, \leq)$.
 \end{theorem} 

The proof of Theorem \ref{wft}
proceeds through several lemmas.
We first recall a result 
from \cite[Theorem 4.7]{w},
which has been mentioned in the previous section
and which we shall repeatedly use here.
Recall the definition of a mixed sum from 
\ref{infmix} and the definition of left-finiteness
from \ref{lf}.  Recall from equation \eqref{idid} that,
for $ \omega$-indexed sequences, $\nsum$
has many equivalent reformulations. 

\begin{theorem} \labbel{wt}
If $( \alpha_i) _{i < \omega} $ is a sequence of ordinals,
then $\nsum _{i < \omega  } \alpha _i$ is 
the largest left-finite mixed sum of   $( \alpha_i) _{i < \omega} $.
  \end{theorem}  

As we mentioned in the previous section, 
Theorem \ref{wt} can be obtained also a 
consequence of Corollary \ref{mixx}. 

\begin{lemma} \labbel{union}
Suppose that $(T, \leq)$ is a countable well-founded tree
and $M$ is the set of the maximal elements of $T$.
Furthermore, suppose that, for every $u \in M$,  
$\leq'_u$ is 
a well-ordered downward-finite
extension of 
$\leq _{{\restriction}{\downarrow}u }$ in ${\downarrow}u$,
and let $\alpha_u$ be the order-type of  $\leq'_u$.
Then $\leq$ has a well-ordered downward-finite
extension $\leq''$ of  order-type
$\nsum _{u \in M} \alpha _u $. 
Moreover, $\leq''$ is such that  
$\leq'' _{{\restriction}{\downarrow}u }$
is equal to
$\leq'_u$,
for every $u \in M$.
 \end{lemma}

 \begin{proof} 
By Theorem \ref{wt}, 
$\nsum _{u \in M} \alpha _u $
is a left-finite mixed sum of 
$( \alpha _u) _{u \in M} $. 
Through the bijections 
from ${\downarrow}u $ to $ \alpha _u$
given by each of the orders $\leq'_u$, 
we can use a realization of  
$\nsum _{u \in M} \alpha _u $
as a left-finite mixed sum of 
$( \alpha _u) _{u \in M} $  
 to construct
a well-order $\leq''$ on $T$  of type 
$\nsum _{u \in M} \alpha _u $
(notice that $T= \bigcup _{u \in M} {\downarrow}u  $,
since the set of successors of each element of $T$
is finite, hence each element of $T$ is $\leq u$,
for some $u \in M$). 
The order $\leq''$ 
is such that 
$\leq'' _{{\restriction}{\downarrow}u }$
is equal to
$\leq'_u$,
for every $u \in M$.
Moreover, $\leq''$  has the following property. 

  \begin{enumerate} \item [(*)]
For every $v \in T$, the set
$\{ u \in M \mid w \leq'' v, \text{ for some }w \leq u  \}$   
is finite.
  \end{enumerate} 

(this is the ``translation'' of left-finiteness
to the new situation, since $w \in {\downarrow}u $ 
if and only if $w \leq u $).

Moreover, 
$\leq''$ extends $\leq$,
since, by assumption,
for every $u \in M$, the order  $\leq'_u$ extends 
$\leq _{{\restriction}{\downarrow}u }$
and since, for $u \not= u ^* \in M$,
all the elements from
${\downarrow}u $ are $\leq$-incomparable
with all the elements from   
${\downarrow}u ^*$.  

Hence it remains to show that
$\leq''$ is a downward-finite
extension of $\leq$.

So let $v \in T$, hence $v \in {\downarrow}u $,
for some $u \in M$.
Since
$\leq'_u$ is 
a downward-finite
extension of 
$\leq _{{\restriction}{\downarrow}u }$,
there are elements
$v_0, \dots v_n \in {\downarrow}u $ 
which are $\leq $-incomparable with $v$ and
 such that
${\downarrow}_u'v \subseteq {\downarrow}v \cup 
{\downarrow}v_0 \cup \dots \cup {\downarrow}v_n$,
where, obviously,
${\downarrow}_u'v $ is computed using
$\leq'_u$ in ${\downarrow}u$, and, since $v, v_0, \dots v_n \in {\downarrow}u $,
then
applying ${\downarrow} $ to $ v $, $ v_0 $, \dots gives the same result
whether computed in  
$({\downarrow}u , \leq _{{\restriction}{\downarrow}u})$
or in 
$(T, \leq)$, hence the notation is not ambiguous. 
Similarly, the incomparabilities of $v $ and  $ v_0 $, etc., 
are equivalently evaluated using 
 $\leq _{{\restriction}{\downarrow}u }$ or
$\leq$.

By (*), and now working in $T$,
there are finitely many elements 
$u_0, \dots, u_m $ in $M$ such that  
${\downarrow}''v  \subseteq 
{\downarrow}u \cup 
{\downarrow}u_0 \cup \dots \cup {\downarrow}u_m$. 
Of course, we can assume that 
$u_0 \not= u$, \dots,
 $u_m \not= u$,
hence the 
$u_ h$'s are $\leq$-incomparable with $v$,
since $v \leq u$, 
since the set of all the successors of $v$ is linearly ordered and since,
for each index $h$,  
we have that $u$ and   $u_ h$'s are distinct 
maximal elements of $T$, hence incomparable.  
Now, 
${\downarrow}''v  \cap 
{\downarrow}u = 
{\downarrow}_u'v $, since 
$\leq'' _{{\restriction}{\downarrow}u }$
is equal to $\leq'_u$, which  
extends 
$\leq _{{\restriction}{\downarrow}u }$
in ${\downarrow}u $.
In conclusion, 
 ${\downarrow}''v  \subseteq 
{\downarrow}_u'v \cup 
{\downarrow}u_0 \cup \dots \cup {\downarrow}u_m
\subseteq
{\downarrow}v \cup 
{\downarrow}v_0 \cup \dots \cup {\downarrow}v_n
\cup
{\downarrow}u_0 \cup \dots \cup {\downarrow}u_m
$,
with $v_0, \dots, v_n,
u_0, \dots, u_m$ all
$\leq$-incomparable with $v$,   what we had to show.
\end{proof}

\begin{lemma} \labbel{elt}
If $(T, \leq)$ is a countable well-founded  tree, then,
for every $t \in T$, 
the restriction $\leq _{{\restriction}{\downarrow}t }$ 
of $\leq$ to 
${\downarrow}t $ has a well-ordered downward-finite extension 
of order-type $\sigma(t)$.
 \end{lemma}

 \begin{proof}
The proof is  by induction of $\rho(t)$.

The base step $\rho(t)=0$ is trivial,
since in this case  $| {\downarrow}t |=1= \sigma(t)$.

Suppose that $\rho(t)  > 0$ and that
the lemma holds for every $u \in T$ with $\rho(u) < \rho (t)$.
In particular, the lemma holds for every 
$u \in P(t)$, where $P(t)$ denotes the set of all
the immediate predecessors  of $t$.
Thus, for every $u \in P(t)$,
$\leq _{{\restriction}{\downarrow}u}$ 
has a well-ordered downward-finite extension $\leq'_u$
on  ${\downarrow}u$ 
of order-type $\sigma(u)$.
 
Let $T^*= \{ v \in T \mid v < t\}  $.
Notice that 
 $T^* =  \bigcup _{ u \in P(t)} {\downarrow}u $,
since the elements of $P(t)$ are the immediate predecessors of $t$ and,
if $v < t$, then $v \leq u$, for some $u \in P(t)$,
since, by the definition of a well-founded tree,
  the successors of $v$  form a finite linearly ordered set.  

We can now apply Lemma \ref{union}
to $T^*$,
getting   
a well-ordered downward-finite extension 
$\leq''$ 
of $\leq  _{{\restriction}{T^*}}$ on
$T^*$ in such
a way that 
$\leq''$
has order-type
$\nsum _{ u\in P(t)}\sigma(u)$.

But then $\leq''$ can be obviously extended
 to an order $\leq'''$ on the whole of ${\downarrow}t$
by putting $t$ on the top. 
Trivially $\leq'''$ is a  well-ordered downward-finite extension
of $\leq$,
and $\leq'''$ has order-type
$ \left(\nsum _{ u \in P(t)}\sigma(u)\right) + 1 = \sigma(t)$.
 \end{proof}

\begin{proof}[Proof of the first sentence in Theorem \ref{wft}]
If $T$ has only one root, then
the result is immediate by applying Lemma \ref{elt}
to this unique root.

The general case follows from the previous case and Lemma \ref{union}.  
\end{proof}

\begin{lemma} \labbel{union2}
Suppose that $(T, \leq)$ is a countable well-founded tree,
$M$ is the set of the maximal elements of $T$
and 
$\leq'$ is 
a well-ordered downward-finite
extension of  $\leq$.

If $\alpha$ is the order-type of $\leq'$ and,
for $u \in M$, 
 $\alpha_u$ is the order-type of 
$\leq' _{{\restriction}{\downarrow}u }$ in ${\downarrow}u$,
then $\alpha$ is a left-finite mixed sum of the 
$\alpha_u$'s. 
 \end{lemma}

\begin{proof}
Since $T= \bigcup _{u \in M} {\downarrow}u  $,
then $\alpha$ is obviously 
a  mixed sum of the 
$\alpha_u$'s. 
Indeed, if $\varphi$  is the bijection
from $T$ onto $\alpha$ induced by
$\leq'$, then, defining  $A_u = \varphi ( {\downarrow}u )$,
for $u \in M$, we get that 
$( A_u) _{u \in M} $ is an appropriate realization of $\alpha$.  

It remains to show that
$( A_u) _{u \in M} $ is a left-finite realization,
but this follows easily from the assumption that
$\leq'$ is 
a downward-finite
extension of 
$\leq$. Indeed, for every $v \in T$, 
there are elements
$v_0, \dots v_n \in T $ 
 such that
${\downarrow}'v \subseteq {\downarrow}v \cup 
{\downarrow}v_0 \cup \dots \cup {\downarrow}v_n$.
But $v \leq u$, for some
(actually, a unique) $u \in M$; similarly,
  $v_0 \leq u_0$, for some
$u_0 \in M$, etc.
Hence
${\downarrow}'v \subseteq 
{\downarrow}v \cup 
{\downarrow}v_0 \cup \dots \cup {\downarrow}v_n
\subseteq
{\downarrow}u \cup 
{\downarrow}u_0 \cup \dots \cup {\downarrow}u_n
$
(repetitions are possible, but they cause no trouble).

Thus if $a \in \alpha$ 
and $a= \varphi(v)$,  
then the set of the elements smaller than $a$ in $\alpha$ 
is contained in the finite union
$A _{u} \cup A _{u_0} \cup  \dots \cup A _{u_n}  $.  
Since $\varphi$  is surjective, this holds for every 
$a \in \alpha$, that is, the realization is left-finite. 

Notice that in the above proof
we do not need the assumption
that the $v_h$'s  are $\leq $-incomparable with $v$.
\end{proof}

\begin{lemma} \labbel{elt2}
Suppose that $(T, \leq)$ is a countable well-founded  tree
and
$\leq'$ is 
a well-ordered downward-finite
extension of  $\leq$.
Then,
for every $t \in T$, 
the restriction $\leq' _{{\restriction}{\downarrow}t }$ 
of $\leq'$ to 
${\downarrow}t $ has  
of order-type $ \leq \sigma(t)$.
 \end{lemma}

 \begin{proof}
By induction of $\rho(t)$.

The base step $\rho(t)=0$ is trivial,
since in this case  $| {\downarrow}t |=1= \sigma(t)$.

Suppose that $\rho(t)  > 0$ and that
the lemma holds for every $u \in T$ with $\rho(u) < \rho (t)$.
In particular, the lemma holds for every 
$u \in P(t)$, where $P(t)$ is the set of all
the immediate predecessors  of $t$.
Thus, for every $u \in P(t)$,
if $\alpha_u$ is the order-type of  
$\leq' _{{\restriction}{\downarrow}u}$, 
then $ \alpha _u \leq \sigma(u)$.
 
Letting $T^*= \{ v \in T \mid v < t\}  $,
we have that 
$\leq' _{{\restriction}T^*}$ 
is 
a well-ordered downward-finite
extension of
$\leq _{{\restriction}T^*}$,
since, by the very definition of ${\downarrow}t$,
$t$ is comparable with every element of 
${\downarrow}t$, and 
comparable elements are not allowed
in the definition  of a downward-finite extension,
Definition \ref{otc}. 

Since, as we noticed in the proof of 
\ref{elt},  
 $T^* =  \bigcup _{ u \in P(t)} {\downarrow}u $,
then we can apply Lemma \ref{union2}.
Hence, if $\alpha$ is the order-type of 
$\leq' _{{\restriction}T^*}$, then
$\alpha$ is a left-finite mixed sum of 
$( \alpha _u) _{u \in P(t)} $.
Since $ \alpha _u \leq \sigma(u)$,
for every $u \in P(t)$, and  using Theorem \ref{wt} and
Proposition \ref{factsww}(2),
we get 
$\alpha \leq 
\nsum _{ u \in P(t)} \alpha _u
\leq
\nsum _{ u \in P(t)}\sigma(u)
$.
Clearly, the order-type of
$\leq' _{{\restriction}{\downarrow}t}$
is $\alpha+1$  and we are done, since, by above,
$\alpha + 1 \leq 
\left(\nsum _{ u \in P(t)}\sigma(u)\right) + 1 = \sigma(t)$.
\end{proof}
 
\begin{proof}[Proof of the last sentence in Theorem \ref{wft}]
If $T$ has only one root, 
the result is immediate from Lemma \ref{elt2}.

The general case follows from the previous case, Lemma \ref{union}
and again  Theorem \ref{wt}.
\end{proof}

\begin{acknowledgement} 
We thank an anonymous referee of \cite{w} for many interesting suggestions
 concerning the relationship between
natural sums and the theory of well-quasi-orders. 
We thank Harry Altman for stimulating discussions.
 \end{acknowledgement}

\begin{disclaimer}
Though the author has done his best efforts to compile the following
list of references in the most accurate way,
 he acknowledges that the list might 
turn out to be incomplete
or partially inaccurate, possibly for reasons not depending on him.
It is not intended that each work in the list
has given equally significant contributions to the discipline.
Henceforth the author disagrees with the use of the list
(even in aggregate forms in combination with similar lists)
in order to determine rankings or other indicators of, e.~g., journals, individuals or
institutions. In particular, the author 
 considers that it is highly  inappropriate, 
and strongly discourages, the use 
(even in partial, preliminary or auxiliary forms)
of indicators extracted from the list in decisions about individuals (especially, job opportunities, career progressions etc.), attributions of funds, and selections or evaluations of research projects.
\end{disclaimer}

This is a preliminary version, it might contain inaccuraccies (to be
precise, it is more likely to contain inaccuracies than planned subsequent versions).
We have not yet performed a completely accurate
search in order to check whether some of the results presented here
are already known. Credits for already known results should go to the
original discoverers.
The above applies in particular with respect to Section \ref{not}.

\end{document}